\documentclass[10pt]{amsart}

\usepackage{amsmath,amsfonts,amssymb}
\usepackage{graphicx,verbatim}
\usepackage[numbers,square,sort&compress]{natbib}
\usepackage[colorinlistoftodos]{todonotes}
\usepackage{times}
\usepackage{xspace}
\usepackage{xypic}
\usepackage{color}
\usepackage{a4wide}
\usepackage{tikz}
\usetikzlibrary{arrows}
\usetikzlibrary{decorations.markings}
\usepackage{subfig}
\usepackage{enumerate}

\graphicspath{{./}{./figs/}{../../figs/}}

\theoremstyle{plain}
  \newtheorem{thm}{Theorem}[section]
  \newtheorem{prop}[thm]{Proposition}
  \newtheorem{defn-prop}[thm]{Definition-Proposition}
  \newtheorem{lem}[thm]{Lemma}
  \newtheorem{cor}[thm]{Corollary}

\theoremstyle{definition}
  \newtheorem{defn}[thm]{Definition}
  \newtheorem{ex}[thm]{Example}
  \newtheorem{rem}[thm]{Remark}

\def\setmin{{-}}

\def\R{\mathbb{R}}
\def\Z{\mathbb{Z}}

\def\AAA{\mathcal{A}}
\def\LLL{\mathcal{L}}

\DeclareMathOperator{\Acyc}{Acyc}
\DeclareMathOperator{\Hasse}{Hasse}
\DeclareMathOperator{\tor}{tor}
\DeclareMathOperator{\torHasse}{torHasse}
\DeclareMathOperator{\aff}{aff}
\DeclareMathOperator{\extreme}{ex}
\DeclareMathOperator{\Chambers}{Cham}

\def\epsilon{\varepsilon}

\tikzset{big arrow/.style={decoration={markings,mark=at position 0.93
      with {\arrow[scale=2]{>}}}, postaction={decorate}, shorten >=8pt}}

\tikzset{big dashed arrow/.style={decoration={markings,mark=at position 0.93
      with {\arrow[scale=2]{>}}}, postaction={decorate}, shorten >=8pt, dashed}}

\newcommand\modone[1]{{{#1} \bmod 1}}

\begin{document}
%%==============

\title{Toric partial orders}

\author[M.~Develin]{Mike Develin} \address{
Department of Audience Research, Facebook Inc.\\
1601 Willow Road\\
Menlo Park, CA 94025}
\email{develin@post.harvard.edu}

\author[M.~Macauley]{Matthew Macauley} \address{Department of
  Mathematical Sciences \\ Clemson University \\ Clemson, SC 29634-0975, USA}
\email{macaule@clemson.edu}

\author[V.~Reiner]{Victor Reiner} \address{Department of Mathematics
  \\ University of Minnesota \\ Minneapolis, MN 55455, USA}
\email{reiner@math.umn.edu}

\thanks{First author supported by AIM Five-Year Fellowship
  (2003--2008). Second author gratefully acknowledges support from
  a Simons Foundation Collaboration Grant \#246042 and 
  NSF grant DMS-1211691. Third author supported by NSF grant DMS-1001933.}

\keywords{Braid arrangement, convex geometry, cyclic order, partial
  order, unimodular, toric arrangement, transitivity, Coxeter element, 
  reflection functor} 

\subjclass[2010]{06A06,52C35}

%% 06A06 Order, lattices, ordered algebraic structures: Partial order, general
%% 52C35 Convex & discrete geometry: Arrangements of points, flats, hyperplanes

\begin{abstract}
  We define toric partial orders, corresponding to regions of graphic
  toric hyperplane arrangements, just as ordinary partial orders
  correspond to regions of graphic hyperplane
  arrangements. Combinatorially, toric posets correspond to finite posets
  under the equivalence relation generated by converting minimal
  elements into maximal elements, or sources into sinks. We derive
  toric analogues for several features of ordinary partial orders,
  such as chains, antichains, transitivity, Hasse diagrams, linear
  extensions, and total orders.
\end{abstract}

\maketitle

%%============================================================
\section{Introduction}
%%============================================================

We define finite {\it toric partial orders} or {\it toric posets}, which are cyclic analogues of partial orders, but differ from an established notion of {\it partial cyclic orders} already in the literature; see Remark~\ref{disambiguation-remark} below.  
Toric posets can be defined in combinatorial geometric ways
that are analogous to partial orders or posets:
\begin{enumerate}
\item[$\bullet$] Posets on a finite set $V$ correspond to open polyhedral cones that arise as chambers in {\it graphic hyperplane arrangements} in $\R^V$; 
toric posets correspond to chambers occurring within 
{\it graphic toric hyperplane arrangements} in the quotient space $\R^V/\Z^V$.
\item[$\bullet$] Posets correspond to {\it transitive closures} of acyclic orientations of graphs;  toric posets correspond to a 
notion of {\it toric transitive closures} of acyclic orientations.  
\item[$\bullet$] Both transitive closure and toric transitive closure
will turn out to be {\it convex closures}, so that there is a notion of {\it toric Hasse diagram}
for a toric poset, like the Hasse diagram of a poset.
\end{enumerate}
We next make this more precise, 
indicating where the main results will be proven.

%%-------------------------------
\subsection{Posets geometrically}
%%-------------------------------
We first recall
(e.g. from Stanley \cite{Stanley:73},
Greene and Zaslavsky \cite[\S 7]{Greene:83}, 
Postnikov, Reiner and Williams \cite[\S\S 3.3-3.4]{Postnikov:08})
geometric features of posets, specifically their relations
to graphic hyperplane arrangements and acyclic orientations, emphasizing 
notions with toric counterparts.

Let $V$ be a finite set of cardinality $|V|=n$;  often we will choose $V=[n]:=\{1,2,\ldots,n\}$.  One can think of a {\it partially ordered
set} or {\it poset} $P$ on $V$ as a binary relation $i <_P j$ which is 
\begin{itemize}
\item {\it irreflexive}: $ i\not <_P i$,
\item {\it antisymmetric}: $i <_P j$ implies $ j \not <_P i $, and 
\item {\it transitive}: $i <_P j$ and $j <_P k$ implies $ i <_P k$.
\end{itemize}
However, one can also identify $P$ with a certain {\it open polyhedral cone} in $\R^V$
\begin{equation}
\label{cone-of-a-poset-defn}
c=c(P):=\{x \in \R^V: x_i < x_j \text{ if }i <_P j\}.
\end{equation}
Note that the cone $c$ determines the poset $P=P(c)$ as follows: $i <_P j$ if and only $x_i < x_j$ for all $x$ in $c$. 

Each such cone $c$ also arises as a connected component in the complement
of at least one {\it graphic hyperplane arrangement} for a graph $G$, and 
often arises in several such arrangements, as explained below.
Given a simple graph $G=(V,E)$,
the {\it graphic arrangement} $\AAA(G)$ 
is the union of all hyperplanes in $\R^V$ of the form 
$x_i=x_j$ where $\{i,j\}$ is in $E$.  Each point $x=(x_1,\ldots,x_n)$ in the {\it complement}
$\R^V \setmin \AAA(G)$ determines an {\it acyclic orientation} $\omega(x)$ 
of the edge set $E$:  for an edge $\{i,j\}$ in $E$, since
$x_i \neq x_j$, either 
\begin{itemize}
\item  $x_i < x_j$ and $\omega(x)$ directs $i \rightarrow j$, or
\item $x_j < x_i$ and $\omega(x)$ directs $j \rightarrow i$.
\end{itemize}
It is easily seen that the fibers of this map $\alpha_G: x \longmapsto \omega(x)$
are the connected components of the complement $\R^V \setmin \AAA(G)$,
which are open polyhedral cones called {\it chambers}.   Thus the map $\alpha_G$
induces a bijection between the set $\Acyc(G)$ 
of all acyclic orientations $\omega$
of $G$ and the set $\Chambers \AAA(G)$ of chambers $c$ of $\AAA(G)$:
\begin{equation}
\label{alpha-diagram}
\xymatrix{
\R^V \setmin \AAA(G) \ar@{>>}[dr] \ar[rr]^{\alpha_G} & & \Acyc(G) \\
 & \Chambers \AAA(G) \ar@{.>}[ur] &\\ 
}
\end{equation}
These two sets are well-known \cite[Theorem 7.1]{Greene:83}, \cite{Stanley:73} 
to have cardinality 
$$
|\Acyc(G)|=|\Chambers \AAA(G)|=T_G(2,0)
$$ 
where $T_G(x,y)$ is the {\it Tutte polynomial} of $G$ \cite{Tutte:54}.  

Posets are also determined by their extensions to
{\it total orders} $w_1 < \cdots < w_n$, which are indexed by 
permutations $w=(w_1,\ldots,w_n)$ of $V$.  The total orders index the chambers 
$$
c_w :=\{x \in \R^V: x_{w_1} < x_{w_2} < \cdots < x_{w_n}\}
$$
in the complement of the {\it complete graphic arrangement $\AAA(K_V)$},
also known as the {\it reflection arrangement of type $A_{n-1}$} or {\it braid arrangement}.
Given a poset $P$, its set $\LLL(P)$ of all {\it linear extensions} or
{\it extensions to a total order}
has the property that 
$$
\overline{c(P)} = \bigcup_{w \in \LLL(P)} \overline{c}_w
$$
where $\overline{(\cdot)}$ denotes topological closure.  Thus when one 
{\it fixes} the graph $G$, chambers $c$ (or posets $P(c)$)
arising as $\alpha_G^{-1}(\omega)$ for various $\omega$ in $\Acyc(G)$
are determined by their sets $\LLL(P(c))$ of linear extensions.

The same poset $P$ or chamber $c=c(P)$ generally 
arises in many graphic arrangements $\AAA(G)$,
as one varies the graph $G$, leading to 
ambiguity in its labeling by a pair $(G,\omega)$ with
$\omega$ in $\Acyc(G)$. Nevertheless, this ambiguity
is well-controlled, in that there are two canonical choices 
$
(\bar{G}(P),\bar{\omega}(P))
$
and
$
(\hat{G}^{\Hasse}(P),\omega^{\Hasse}(P))
$
with the following properties.
\begin{enumerate}
\item[$\bullet$] A graph $G$ has $c(P)$ occurring
in $\Chambers \AAA(G)$
if and only if
$
\hat{G}^{\Hasse}(P) \subseteq G  \subseteq \bar{G}(P)
$
where $\subseteq$ is inclusion of edge sets.
In this case, $\alpha_G(c(P))=\omega$ where 
$\omega$ is the restriction $\bar{\omega}(P)|_G$.

\item[$\bullet$] The map which sends 
$(G,\omega) \longmapsto (\bar{G}(P),\bar{\omega}(P))$ is {\it transitive closure}.
It adds into $G$ all edges $\{i,j\}$ which lie on some {\it chain} 
(= {\it totally ordered subset}) $C$ of $P$,
and directs $i \rightarrow j$ if $i <_C j$.
Alternatively phrased, transitive closure adds the
directed edge $i \rightarrow j$ to $(G,\omega)$ whenever there is
a directed path from $i$ to $j$ in $(G,\omega)$.
\end{enumerate}
The existence of a unique
{\it inclusion-minimal} choice $(\hat{G}^{\Hasse}(P),\omega^{\Hasse}(P))$, 
called the {\it Hasse diagram} for $P$, follows 
from this well-known fact \cite{Edelman:82, Edelman:02}:
the {\it transitive closure} $A \longmapsto \bar{A}$
on subsets $A$ of all possible oriented edges 
$
\overleftrightarrow{K}_V=\{ (i,j) \in V \times V : i \neq j \},
$  
is a {\it convex closure}, meaning that
\begin{equation}
\label{convex-closure-definition}
\text{ for } a \neq b \text{ with }
a,b \not\in \bar{A}\text{ and }
a \in \overline{A \cup \{b\}},
\text{ one has }b \notin \overline{A \cup \{a\}}.
\end{equation}

%%-----------------------
\subsection{Toric posets}
%%-----------------------

We do not initially define a toric poset $P$ on the finite set $V$
via some binary (or ternary) relation. Rather we 
define it in terms of chambers in a {\it toric graphic arrangement}
$\AAA_{\tor}(G)=\pi(\AAA(G))$,
the image of the graphic arrangement $\AAA(G)$ under the
quotient map $\R^V \overset{\pi}{\rightarrow} \R^V/\Z^V$.  These are
important examples of {\it unimodular toric arrangements} discussed by
Novik, Postnikov and Sturmfels in \cite[\S\S 4-5]{Novik:02}; see
also Ehrenborg, Readdy and Slone \cite{Ehrenborg:09}.

\begin{defn}
\label{toric-poset-defn}
A connected component $c$ of the complement 
$\R^V/\Z^V \setmin \AAA_{\tor}(G)$ is called a {\it toric chamber} for $G$;  
denote by $\Chambers \AAA_{\tor}(G)$
the set of all toric chambers of $\AAA_{\tor}(G)$.  

A {\it toric poset} $P$ is a set $c$ that arises as a toric chamber for
at least one graph $G$.  We will write $P=P(c)$ and $c=c(P)$, depending
upon the context.
\end{defn}

\begin{ex}
When $n=2$, so $V=\{1,2\}$, 
there are only two simple graphs $G=(V,E)$,
a graph $G_0$ with no edges and the complete graph $K_2$ with
a single edge $\{1,2\}$.
For both such graphs, the torus  $\R^2/\Z^2$ remains connected
after removing the arrangement $\AAA_{\tor}(G)$, and hence
they each have only one toric chamber; call these chambers $c_0(=\R^2/\Z^2)$
for the graph $G_0$, and $c(=\R^2/\Z^2\setmin \{x_1=x_2\})$
for the graph $K_2$.  They represent two different toric posets $P(c_0)$
and $P(c)$,
even though their topological closures $\bar{c}=\bar{c}_0(=c_0)=\R^2/\Z^2$
are the same.
\end{ex}

A point $x$ in $\R^V/\Z^V$ does not have uniquely defined coordinates 
$(x_1,\ldots,x_n)$.  However, it is well-defined to speak of the
{\it fractional part} $\modone{x_i}$,
that is, the unique representative of the class of $x_i$ in $\R/\Z$ that 
lies in $[0,1)$.  Therefore a point $x$ in $\R^V/\Z^V \setmin \AAA_{\tor}(G)$,
still induces an acyclic orientation $\omega(x)$ of $G$, as follows:
for each edge $\{i,j\}$ in $E$, since
$x_i \neq x_j \bmod{\Z}$, either 
\begin{itemize}
\item  $\modone{x_i} < \modone{x_j}$, and $\omega(x)$ directs $i \rightarrow j$, or
\item $\modone{x_j} < \modone{x_i}$, and $\omega(x)$ directs $j \rightarrow i$.
\end{itemize}
Denote this map $x \mapsto \omega(x)$ by
$
\R^V/\Z^V \setmin \AAA_{\tor}(G)
 \overset{\bar{\alpha}_G}{\longrightarrow}
   \Acyc(G).
$
Unfortunately, two points lying in the same toric chamber 
$c$ in $\Chambers_{\tor} \AAA_{\tor}(G)$ need not
map to the same acyclic orientation under $\bar{\alpha}_G$.
This ambiguity leads one naturally to the 
following equivalence relation on acyclic orientations.

\begin{defn}
  When two acyclic orientations $\omega$ and $\omega'$ of $G$ differ
  only by converting one source vertex of $\omega$ into a sink of $\omega'$, 
  say that they differ by a \emph{flip}.
  The transitive closure of the flip operation generates an
  equivalence relation on $\Acyc(G)$ denoted by $\equiv$.
\end{defn}

A thorough investigation of this source-to-sink flip operation and
equivalence relation was 
undertaken by Pretzel in~\cite{Pretzel:86}, and
studied earlier by Mosesjan \cite{Mosesjan:72}.
It has also appeared at other times in various contexts\footnote{
Pretzel called the source-to-sink flip {\it pushing down maximal
vertices}; in~\cite{Macauley:11}, it was called a \emph{click}. 
In the category of representations of a quiver, it is 
related to Bernstein, Gelfand and Ponomarev's 
{\it reflection functors} \cite{Bernstein:73}.}
in the literature~\cite{Chen:10,Eriksson:09,Macauley:11,Speyer:09}.
Its relation to geometry of toric chambers $c=c(P)$ or
toric posets $P=P(c)$ is our first main result, 
proven in \S~\ref{geometry-section}.

\begin{thm}
\label{geometric-interpretation-theorem}
The map $\bar{\alpha}_G$ induces a bijection
between $\Chambers \AAA_{\tor}(G)$ and $\Acyc(G)/\!\!\equiv$ as follows:

\begin{equation}
\label{alpha-floor-diagram}
\xymatrix{
\R^V/\Z^V \setmin \AAA_{\tor}(G) 
  \ar@{>>}[d] \ar[r]^{\bar{\alpha}_G} & \Acyc(G)  \ar@{>>}[d] \\
 \Chambers \AAA_{\tor}(G) \ar@{.>}[r]_{\bar{\alpha}_G} & \Acyc(G)/\!\!\equiv
}
\end{equation}
In other words, two points $x,x'$ in $\R^V/\Z^V \setmin \AAA_{\tor}(G)$
have $\bar{\alpha}_G(x)\equiv\bar{\alpha}_G(x')$ if and only if
$x,x'$ lie in the same toric chamber $c$ in $\Chambers \AAA_{\tor}(G)$.
\end{thm}

\noindent
The two sets $\Chambers \AAA_{\tor}(G)$ and $\Acyc(G)/\!\!\equiv$ appearing
in the theorem are known to have cardinality
$$
|\Acyc(G) / \!\! \equiv |=|\Chambers \AAA_{\tor}(G)|=T_G(1,0)
$$ 
where $T_G(x,y)$ is the Tutte polynomial of $G$;
see \cite{Macauley:08} and \cite[Theorem 4.1]{Novik:02}.

\begin{ex}
A {\it tree} $G$ on $n$ vertices has Tutte polynomial
$T_G(x,y)=x^{n-1}$.  It will have $T(2,0)=2^{n-1}$ 
acyclic orientations $\omega$ and induced partial orders,
but only $T(1,0)=1$ toric chamber or toric partial order:
any two acyclic orientations of a tree are equivalent by a 
sequence of source-to-sink moves.
\end{ex}

\begin{ex}
\label{four-cycle-toric-posets-example}
As a less drastic example, consider 
$V=\{1,2,3,4\}$ and $G=(V,E)$ this graph:
$$
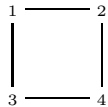
\tiny{
\xymatrix{ 
1\ar@{-}[r] & 2\ar@{-}[d] \\
3\ar@{-}[u]& 4\ar@{-}[l]
}
}
$$
It has Tutte polynomial $T_G(x,y)=x^3+x^2+x+y$, and hence
has $T_G(2,0)=2^3+2^2+2+0=14$ acyclic orientations $\omega$.
These $\omega$ fall into $T_G(1,0)=1^3+1^2+1+0=3$ different $\equiv$-classes 
$[\omega]$, having cardinalities $4,4,6$, respectively,
corresponding to three different toric posets $P_i$ 
or chambers $c_i$ in $\Chambers \AAA_{\tor}(G)$:
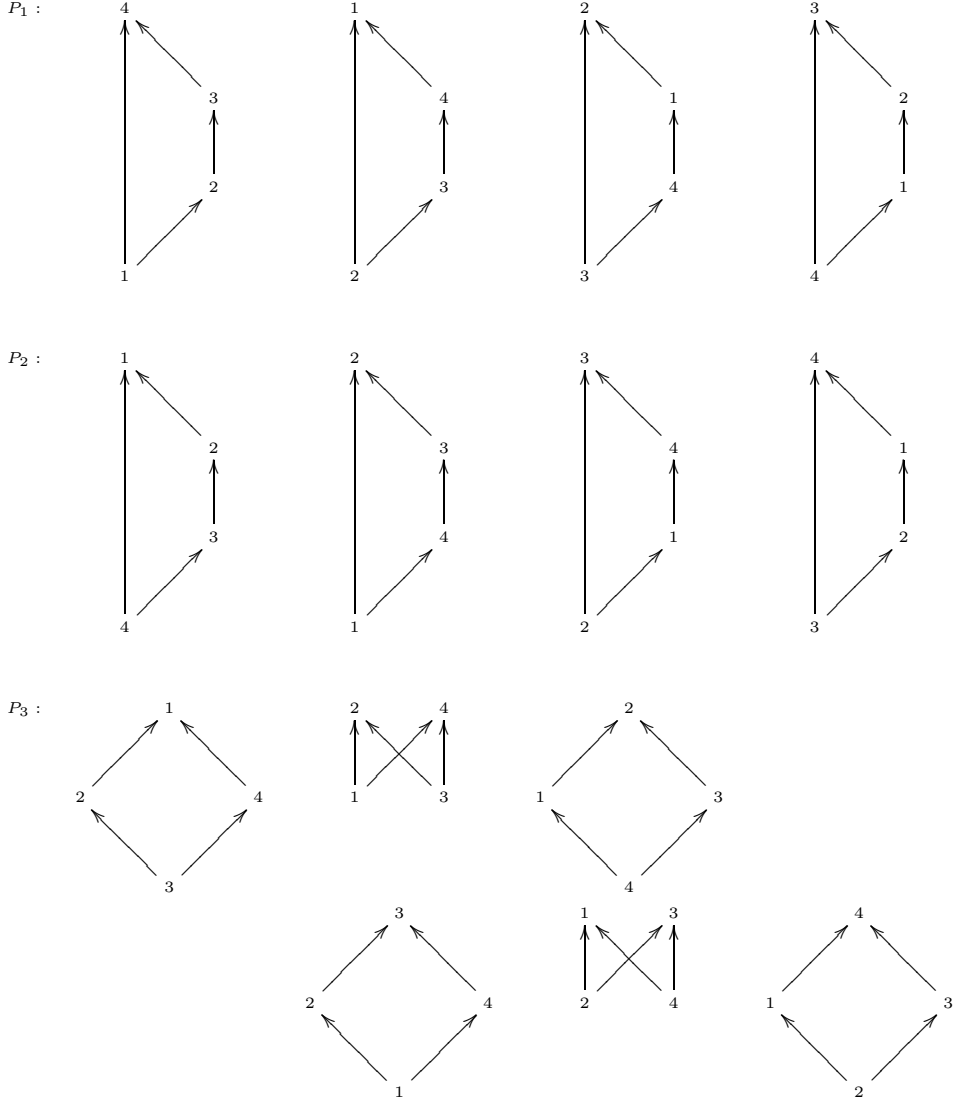
$$
\tiny{
\begin{array}{rcccc}
P_1:
&
\xymatrix{ 
4                & \\
                 & 3 \ar[ul] \\
                 & 2 \ar[u] \\
1\ar[ur] \ar[uuu]
}
&
\xymatrix{ 
1                & \\
                 & 4 \ar[ul] \\
                 & 3 \ar[u] \\
2\ar[ur] \ar[uuu]
}
&
\xymatrix{ 
2                & \\
                 & 1 \ar[ul] \\
                 & 4 \ar[u] \\
3\ar[ur] \ar[uuu]
}
&
\xymatrix{ 
3                & \\
                 & 2 \ar[ul] \\
                 & 1 \ar[u] \\
4\ar[ur] \ar[uuu]
}
\\
%%%%%%%%%%%%%%%%
& & & &  \\
& & & &  \\
& & & &  \\
%%%%%%%%%%%%%%%%
P_2:
&
\xymatrix{ 
1                & \\
                 & 2 \ar[ul] \\
                 & 3 \ar[u] \\
4\ar[ur] \ar[uuu]
}
&
\xymatrix{ 
2                & \\
                 & 3 \ar[ul] \\
                 & 4 \ar[u] \\
1\ar[ur] \ar[uuu]
}
&
\xymatrix{ 
3                & \\
                 & 4 \ar[ul] \\
                 & 1 \ar[u] \\
2\ar[ur] \ar[uuu]
}
&
\xymatrix{ 
4                & \\
                 & 1 \ar[ul] \\
                 & 2 \ar[u] \\
3\ar[ur] \ar[uuu]
}
\\
%%%%%%%%%%%%%
& & & &  \\
& & & &  \\
& & & &  \\
%%%%%%%%%%%%%
P_3:
&
\xymatrix{
        & 1              & \\
2\ar[ur]&                &4\ar[ul] \\
        & 3\ar[ur]\ar[ul]&
}
& 
\xymatrix{
 2              & 4 \\
 1\ar[u]\ar[ur] & 3\ar[u]\ar[ul]
}
&
\xymatrix{
        & 2              & \\
1\ar[ur]&                &3\ar[ul] \\
        & 4\ar[ur]\ar[ul]&
}
& \\
 & & 
\xymatrix{
        & 3              & \\
2\ar[ur]&                &4\ar[ul] \\
        & 1\ar[ur]\ar[ul]&
}
& 
\xymatrix{
 1              & 3 \\
 2\ar[u]\ar[ur] & 4\ar[u]\ar[ul]
}
&
\xymatrix{
        & 4              & \\
1\ar[ur]&                &3\ar[ul] \\
        & 2\ar[ur]\ar[ul]&
} 
\end{array}
}
$$

\end{ex}

{\it Toric total orders} (see \S~\ref{total-orders-section}) 
are indexed by the $(n-1)!$ cyclic equivalence classes of permutations 
\begin{equation}
\begin{array}{rcl}
\label{typical-cyclic-permutation}
[w]:=[(w_1,w_2,\ldots,w_n)]=\left\{\right. &(w_1,w_2,\ldots,w_{n-1},w_n), & \\
             &(w_2,\ldots,w_{n-1},w_n,w_1), &\\
             &\vdots&\\
             &(w_n,w_1,w_2,\ldots,w_{n-1}) &\left.\right\}
\end{array}
\end{equation}
and correspond to the toric chambers $c_{[w]}$
in the complement of the {\it toric complete graphic arrangement $\AAA_{\tor}(K_V)$}.
For a particular toric poset $P=P(c)$, 
one says that $[w]$ is a {\it toric total
extension} of $P$ if $c_{[w]} \subseteq c$.  Denote by $\LLL_{\tor}(P)$ the
set of all such toric total extensions $[w]$ of $P$.  
Although it is possible (see Example~\ref{lack-of-determination-by-extensions-example} below)
for two different toric posets $P$ to have the same set $\LLL_{\tor}(P)$,
the following assertion (combining 
Proposition~\ref{edge-subgraph-chamber-refinement} and
Corollary~\ref{determination-by-total-cyclic-extensions} below) still holds.

\begin{prop}
When one {\it fixes} the graph $G$, the toric chamber $c$ 
(or its poset $P=P(c)$) for which $\bar{\alpha}_G(c)=[\omega]$ 
is completely determined by its topological
closure $\overline{c}$.
Furthermore one has
$
\overline{c} = \bigcup_{w \in \LLL_{\tor}(P)} \overline{c}_{[w]}.
$
so that this closure depends only on the set of toric total extensions
$\LLL_{\tor}(P)$.

\end{prop}

\begin{ex}
The graph $G$ from Example~\ref{four-cycle-toric-posets-example}
and its three toric posets $P_1,P_2,P_3$ partition
the $(4-1)!=6$ different toric total orders on $V=\{1,2,3,4\}$
into their sets of toric total extensions $\LLL_{\tor}(P_i)$ as follows:
$$
\begin{aligned}
\LLL_{\tor}(P_1) 
  & = \{ [(1,2,3,4)] \},\\
\LLL_{\tor}(P_2) 
  & = \{ [(1,4,3,2)] \},\\
\LLL_{\tor}(P_3) 
  & = \{ [(1,2,4,3)],[(1,3,2,4)],  [(1,3,4,2)],[(1,4,2,3)]\}.
\end{aligned}
$$
\end{ex}

As with posets, the same toric poset $P=P(c)$ arises as a chamber $c$
in {\it many} toric graphic arrangements $\AAA_{\tor}(G)$.
However, as with posets, this ambiguity is well-controlled,
in that there are two canonical choices of equivalence classes
$
(\bar{G}^{\tor}(P), [\bar{\omega}^{\tor}(P)])
$
and
$
(\hat{G}^{\torHasse}(P), [\omega^{\torHasse}(P)])
$
with the following properties.
\begin{enumerate}

\item[$\bullet$] A graph $G$ has $c(P)$ occurring in
$\Chambers \AAA_{\tor}(G)$ if and only if
$$
\hat{G}^{\torHasse}(P) \subseteq G  \subseteq \bar{G}^{\tor}(P)
$$
where $\subseteq$ is inclusion of edges.  In this case, 
if $\bar{\alpha}_G(c(P))=[\omega]$, then
$\omega$ can be taken to be the
restriction to $G$ of a particular orientation 
in the class $[\bar{\omega}^{\tor}(P)]$.

\item[$\bullet$] The map which sends 
$(G,\omega) \longmapsto (\bar{G}^{\tor}, \bar{\omega}^{\tor})$ may be described by what
will be called (in \S~\ref{transitivity-section}) {\it toric transitive closure}:
one adds into $G$ all edges $\{i,j\}$ which lie on some {\it toric chain} $C$ in $P$.
Here a toric chain (see \S~\ref{chains-section}) is a subset $C \subset V$
which is totally ordered in {\it every} poset associated with
an orientation in the class $[\omega]$.  
One directs $i \rightarrow j$ if there is a 
{\it toric directed path} from $i$ to $j$ in $(G,\omega)$, as defined 
in \S~\ref{directed-paths-section} below.
Alternatively phrased, toric transitive closure will add the
directed edge $i \rightarrow j$ to $(G,\omega)$ whenever there is
a toric directed path from $i$ to $j$ in $(G,\omega)$.
\end{enumerate}
The existence of the unique
{\it inclusion-minimal} choice $(\hat{G}^{\torHasse}(P),[\omega^{\torHasse}(P)])$, 
which we will call the {\it toric Hasse diagram} of $P$, follows 
from our second main result, proven in \S~\ref{convex-closure-section}.

\begin{thm}
\label{toric-convex-closure-theorem}
Considered as a {\it closure operation} $A \longmapsto \bar{A}^{\tor}$
on subsets $A$ of all possible oriented edges 
$
\overleftrightarrow{K}_V=\{ (i,j) \in V \times V : i \neq j \},
$
toric transitive closure is a {\it convex closure}, that is,
it satisfies \eqref{convex-closure-definition} above.
\end{thm}

\begin{ex}
The toric poset $P_1=P(c_1)$ from Example~\ref{four-cycle-toric-posets-example} 
appears as a chamber $c_1$ in $\Chambers \AAA_{\tor}(G_i)$ for
exactly four graphs $G_1,G_2,G_3,G_4$, each shown below with an
orientation $\omega_i$ such that $\bar{\alpha}_{G_i}(c_1)=[\omega_i]$.
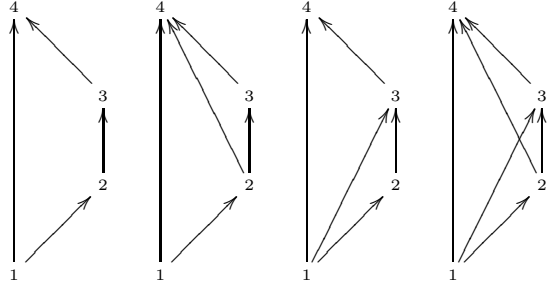
$$
\tiny{
\xymatrix{ 
4                & \\
                 & 3 \ar[ul] \\
                 & 2 \ar[u] \\
1\ar[ur] \ar[uuu]
}
\qquad
\xymatrix{ 
4                & \\
                 & 3 \ar[ul] \\
                 & 2 \ar[u]\ar[uul] \\
1\ar[ur] \ar[uuu]
}
\qquad
\xymatrix{ 
4                & \\
                 & 3 \ar[ul] \\
                 & 2 \ar[u] \\
1\ar[ur] \ar[uuu] \ar[uur]
}
\qquad
\xymatrix{ 
4                & \\
                 & 3 \ar[ul] \\
                 & 2 \ar[u] \ar[uul]\\
1\ar[ur] \ar[uuu] \ar[uur]
}
}
$$
For any of these four pairs $(G_i,\omega_i)$ with $i=1,2,3,4$, one has that the
leftmost pair is its Hasse diagram $(\hat{G_i}^{\torHasse},\omega_i^{\torHasse})$,
and the rightmost pair is its toric transitive closure
$(\bar{G}^{\tor}_i,\bar{\omega}_i^{\tor})$.
\end{ex}

We close this Introduction with two remarks, one on terminology, 
the other giving further motivation.

\begin{rem}
\label{disambiguation-remark}
Aside from the connection to toric hyperplane arrangements,
we have chosen the name ``toric partial order'', 
as opposed to the arguably more natural term 
``cyclic partial order'', because the latter is
easily confused with {\it partial cyclic orders}, the following
pre-existing concept in the literature, going
back at least as far as Megiddo~\cite{Megiddo:76}.
\begin{defn}
  A \emph{partial cyclic order} on $V$ is a ternary relation $T \subseteq V\times V \times V$
  that is  
\begin{itemize}
  \item \emph{antisymmetric}: If $(i,j,k)\in T$ then $(k,j,i)\not\in T$;
  \item \emph{transitive}: If $(i,j,k)\in T$ and $(i,k,\ell)\in T$, then
    $(i,j,\ell)\in T$;
  \item \emph{cyclic}: If $(i,j,k)\in T$, then $(j,k,i)\in T$.
\end{itemize}
\end{defn}

\begin{defn}
\label{total-cyclic-order-defn}
When a partial cyclic order on $V$ is \emph{complete} in the sense that
for every triple $\{i,j,k\} \subseteq V$ of distinct elements, 
$T$ contains some permutation of $(i,j,k)$, then $T$ is called a \emph{total cyclic order}.  A total cyclic order on $V$ is easily seen to be the same a toric total order: specify a cyclic equivalence class $[w]$ as 
in \eqref{typical-cyclic-permutation}, and then check that $[w]$ is
determined by knowing its restrictions $[w|_{\{i,j,k\}}]$
for all triples $\{i,j,k\}$. 
\end{defn}

Partial cyclic orders have been widely studied, and
have some interesting features not shared by ordinary partial
orders. For example, every partial order can be extended to a total
order, but not every partial cyclic order can be extended 
to a total cyclic order;  an example of this on $13$ vertices
is given in \cite{Megiddo:76}.
\end{rem}

\begin{rem}
We mention a further analogy between posets and toric posets,
related to Coxeter groups, that was 
one of our motivations for formalizing this concept.

Recall \cite{Bourbaki:02} that a {\it Coxeter system} $(W,S)$ is a group $W$ 
with generating set $S=\{s_1,\ldots,s_n\}$ having presentation
$
W= \langle S : (s_i s_j)^{m_{i,j}}=e \rangle
$
for some $m_{i,j}$ in $\{1,2,3,\ldots\} \cup \{\infty\}$,
where $m_{i,i}=1$ for all $i$ and $m_{i,j} \geq 2$ for $i \neq j$.
Associated to $(W,S)$ is the {\it Coxeter graph} on vertex set $S$
with an edge $\{s_i,s_j\}$ labeled by $m_{i,j}$ whenever $m_{i,j} > 2$, so that
$s_i,s_j$ do not commute; ignoring the edge labels, we will call this
the unlabeled Coxeter graph.  
A {\it Coxeter element} for $(W,S)$ is an element of
the form $s_{w_1} s_{w_2} \cdots s_{w_n}$ for some choice of a total
order $w$ on $S$.

\begin{thm} Fix a Coxeter system $(W,S)$ with unlabeled Coxeter
graph $G$, and consider the map sending an 
acyclic orientation $\omega$ in $\Acyc(G)$ having poset $P=\alpha_G(\omega)$
to the Coxeter element $s_{w_1} s_{w_2} \cdots s_{w_n}$ for any choice
of a linear extension $w$ in $\LLL(P)$.  
\begin{enumerate}
\item[(i)] 
This map is well-defined, and induces a bijection 
(see \cite[\S V.6]{Bourbaki:02} and \cite{Cartier:69})
$$
\Acyc(G) 
\longleftrightarrow 
\{ \text{ Coxeter elements for }(W,S) \,\, \}.
$$ 
\item[(ii)] 
It also induces a well-defined map on the toric equivalence classes 
$[\omega]$ to the {\bf $W$-conjugacy classes} of all Coxeter elements,
and gives a bijection (see \cite{Eriksson:09, Macauley:08, Macauley:11, Shi:97}
and \cite[Remark 5.5]{Novik:02})
$$
\Acyc(G)/\!\!\equiv 
\quad \longleftrightarrow \quad
\{ W\text{-conjugacy classes of Coxeter elements for }(W,S) \}.
$$ 
\end{enumerate}
\end{thm}
\noindent
We believe toric partial orders will play a key role in resolving 
more questions about $W$-conjugacy classes.

\end{rem}

%%============================================================
\section{Toric arrangements and 
proof of Theorem~\ref{geometric-interpretation-theorem}}
\label{geometry-section}
%%============================================================

Recall the statement of the theorem.

\vskip.1in
\noindent
{\bf Theorem~\ref{geometric-interpretation-theorem}.}
{\it 
The map $\bar{\alpha}_G$ induces a bijection
between $\Chambers \AAA_{\tor}(G)$ and $\Acyc(G)/\!\!\equiv$ as follows:
$$
\xymatrix{
\R^V/\Z^V \setmin \AAA_{\tor}(G) 
  \ar@{>>}[d] \ar[r]^{\bar{\alpha}_G} & \Acyc(G)  \ar@{>>}[d] \\
 \Chambers \AAA_{\tor}(G) \ar@{.>}[r]_{\bar{\alpha}_G} & \Acyc(G)/\!\!\equiv
}
$$
In other words, two points $x,x'$ in $\R^V/\Z^V \setmin \AAA_{\tor}(G)$
have $\bar{\alpha}_G(x)\equiv\bar{\alpha}_G(x')$ if and only if
$x,x'$ lie in the same toric chamber $c$ in $\Chambers \AAA_{\tor}(G)$.
}

\vskip.1in
\noindent
Before embarking on the proof, we introduce one further geometric object
intimately connected with 
\begin{itemize}
\item
the graphic arrangement $\AAA(G)=\bigcup_{\{i,j\} \in E} \{x \in \R^V: x_i = x_j\} \subset \R^V$, and
\item
the toric graphic arrangement 
$\AAA_{\tor}(G)=\pi(\AAA(G))$, its image 
under $\R^V \overset{\pi}{\rightarrow} \R^V/\Z^V$.
\end{itemize}

\begin{defn}
  Define the \emph{affine graphic arrangement} in $\R^V$ by
\begin{equation}
\label{affine-arrangement-defn}
\AAA_{\aff}(G) :=  \pi^{-1}(\AAA_{\tor}(G)) = \pi^{-1}(\pi(\AAA(G)))
=\bigcup_{\substack{ \{i,j\} \in E\\ k \in \Z}} \{x \in \R^V: x_i = x_j+k \}.
\end{equation}
  Call the connected components
  $\hat{c}$ of the complement $\R^V \setmin\AAA_{\aff}(G)$ {\it affine
  chambers}, and denote the set of all such chambers $\Chambers \AAA_{\aff}(G)$.
\end{defn}

The reason for introducing $\AAA_{\aff}(G)$ and $\Chambers \AAA_{\aff}(G)$
is the following immediate consequence of the path-lifting
property for $\R^V \overset{\pi}{\rightarrow} \R^V/\Z^V$ as
a (universal) covering map (see e.g. \cite[Chap. 13]{Munkres:75}), along
with the definition \eqref{affine-arrangement-defn} of $\AAA_{\aff}(G)$ 
as the full inverse image under $\pi$ of $\AAA_{\tor}(G)$.

\begin{prop}
\label{path-lifting-prop}
Two points $x,y$ in $\R^V/\Z^V \setmin \AAA_{\tor}(G)$
lie in the same chamber $c$ in $\Chambers\AAA_{\tor}(G)$ if
and only if they have two lifts $\hat{x}, \hat{y}$ lying in the
same affine chamber $\hat{c}$ in $\Chambers \AAA_{\aff}(G)$.
\end{prop}

\noindent
The point will be that, since affine chambers  $\hat{c}$ are (open) convex 
polyhedral regions in $\R^V$, it is sometimes easier to argue about 
lifted points $\hat{x}$ rather than $x$ itself.

Our proof of 
Theorem~\ref{geometric-interpretation-theorem}
proceeds by showing the map 
$
\R^V/\Z^V \setmin \AAA_{\tor}(G) 
  \overset{\bar{\alpha}_G}{\longrightarrow} \Acyc(G)
$
descends to 
\begin{itemize}
\item a well-defined map 
$
\Chambers \AAA_{\tor}(G) 
  \overset{\bar{\alpha}_G}{\longrightarrow} \Acyc(G)/\!\!\equiv,
$
\item which is surjective, 
\item and injective.
\end{itemize}

%%--------------------------
\subsection{Well-definition}
%%--------------------------

We must show that when $x, y$ lie in the same toric chamber
$c$ in $\Chambers\AAA_{\tor}(G)$, then
$\bar{\alpha}_G(x) \equiv \bar{\alpha}_G(y)$.
As in Proposition~\ref{path-lifting-prop}, pick lifts 
$\hat{x},\hat{y}$ in $\R^V$ and a path $\hat{\gamma}$ between
them in some affine chamber $\hat{c}$.  Because these chambers are
open, one can assume without loss of generality that $\hat{\gamma}$
takes steps in coordinate directions only, and therefore that
$\hat{x}, \hat{y}$ differ in only a single coordinate:  say
$\hat{x}_i \neq \hat{y}_i$, but $\hat{x}_j=\hat{y}_j$ for all $j \neq i$.
Furthermore, as $\bar{\alpha}_G(x)$ changes only when a coordinate
of $\hat{x}$ passes through an integer, without loss of generality, one may assume
$$
\begin{aligned}
\modone{\hat{x}_i}&=1-\epsilon,\\
\modone{\hat{y}_i}&=\epsilon
\end{aligned}
$$ 
for some arbitrarily small $\epsilon > 0$.  
Since the points on $\hat{\gamma}$ all avoid $\AAA_{\aff}(G)$,
and the $i^{th}$ coordinate will pass through $0$ at some point
on the path $\hat{\gamma}$, each of the
coordinates $\hat{x}_j(=\hat{y}_j)$ for indices $j$ with 
$\{i,j\}$ in $E$ must have $0 < \modone{\hat{x}_j} < 1$.  
Hence one can choose $\epsilon$
small enough that all $j$ for which $\{i,j\}$ in $E$ satisfy
$$
\left( \modone{\hat{y}_i} = \right) \epsilon < \modone{\hat{x}_j} 
         < 1 - \epsilon \left( = \modone{\hat{x}_i} \right).
$$
One finds that $\bar{\alpha}_G(\hat{x})$ and $\bar{\alpha}_G(\hat{y})$ 
differ by changing $i$ from sink to a source, so
$\bar{\alpha}_G(\hat{x}) \equiv \bar{\alpha}_G(\hat{y})$,
as desired.

%%-----------------------
\subsection{Surjectivity}
%%-----------------------

It suffices to check that the map 
$
\R^V/\Z^V \setmin \AAA_{\tor}(G) 
  \overset{\bar{\alpha}_G}{\longrightarrow} \Acyc(G)
$
is surjective.  Given an acyclic orientation $\omega$
of $G$, pick any linear extension $w_1 < \cdots < w_n$ of its associated
partial order $\alpha_G^{-1}(\omega)$ on $V$.  Then choose real numbers
$0 < x_{w_1} < \cdots < x_{w_n} < 1$, so that
$$
x=(x_1,\ldots,x_n)=(\modone{x_1},\ldots,\modone{x_n})
$$
and hence $\bar{\alpha}_G(x)=\omega$.

%%----------------------
\subsection{Injectivity}
%%----------------------

The key to injectivity is the following lemma.
  \begin{lem}\label{lem:key}
   Suppose $x$ lies in a toric chamber $c$ in $\Chambers \AAA_{\tor}(G)$,
and $\bar{\alpha}_G(x)=\omega$.  
Then for any $\omega' \equiv \omega$, there exists
some $x'$ in the same toric chamber $c$ having $\bar{\alpha}_G(x')=\omega'$.
  \end{lem}

  \begin{proof}
 It suffices to check this when $\omega'$ is obtained from $\omega$ by
changing a source vertex $i$ in $\omega$ to a sink in $\omega'$.
Since $\bar{\alpha}_G(x)=\omega$, one must have for each $j$ with
$\{i,j\}$ in $E$ that
$$
(0 \leq ) \modone{x_i} < \modone{x_j} (< 1).
$$
Lift $x$ to $\hat{x}=(\modone{x_1},\ldots,\modone{x_n})$, and
choose $\epsilon$ small enough so that each $j$ with $\{i,j\}$ in $E$ has
$
\modone{x_j} < 1-\epsilon.
$
Define $\hat{y}$ to have all the same coordinates as $\hat{x}$ except for
$\hat{y}_i=-\epsilon$, so that $\modone{\hat{y}_i}=1-\epsilon$, and
hence $y:=\pi(\hat{y})$ has $\bar{\alpha}_G(y)=\omega'$ by construction.  
Note that the straight-line path $\hat{\gamma}$
from $\hat{x}$ to $\hat{y}$ changes only the $i^{th}$ coordinate,
decreasing it from $\hat{x}_i$ to $\hat{y}_i=-\epsilon$,
and hence never crosses any of the affine
hyperplanes in $\AAA_{\aff}(G)$.  Therefore $\hat{x},\hat{y}$ lie
in the same affine chamber, and $x,y$ lie in the same toric
chamber $c$.
  \end{proof}

Now suppose that points $x,x'$ in two toric chambers $c,c'$ 
have $\bar{\alpha}_G(x) \equiv \bar{\alpha}_G(x')$, and we must show that
$c=c'$.  By Lemma~\ref{lem:key}, without loss of generality one
has $\bar{\alpha}_G(x) = \omega = \bar{\alpha}_G(x')$.  Thus one can
lift $x,x'$ to $\hat{x},\hat{x}'$ having $\hat{x}_i,\hat{x}'_i$ in $[0,1)$
for all $i$, and hence
$\alpha_G(\hat{x})=\omega=\alpha_G(\hat{x}')$.
For each edge $\{i,j\}$ in $E$, say directed $i\to j$ in $\omega$,
one has both 
$$
\begin{aligned}
0 \leq \hat{x}_i &< \hat{x}_j <1,\\ 
0 \leq \hat{x}'_i &< \hat{x}'_j <1.
\end{aligned}
$$
Thus every point $\hat{y}$ on the straight-line path 
$\hat{\gamma}$ between $\hat{x}$ and $\hat{x}'$ also satisfies
$0 \leq \hat{y}_i < \hat{y}_j < 1$, avoiding all affine 
hyperplanes in $\AAA_{\aff}(G)$.  Thus $\hat{x},\hat{x}'$ lie 
in the same affine chamber $\hat{c}$, so that $x,x'$ lie
in the same toric chamber, as desired.  This completes the proof
of injectivity, and hence the proof of Theorem~\ref{geometric-interpretation-theorem}.$\qed$

\vskip.2in
One corollary to Theorem~\ref{geometric-interpretation-theorem} is a (slightly)
more concrete description of a toric chamber $c$.

\begin{cor}
\label{toric-chamber-as-union-cor}
For a graph $G=(V,E)$ and toric chamber $c$ in $\Chambers \AAA_{\tor}(G)$
with $\bar{\alpha}_G(c)=[\omega]$, one has
$$
c= \bigcup_{ \omega' \in [\omega] }
       \bar{\alpha}_G^{-1}(\omega') 
= \bigcup_{ \omega' \in [\omega]  }
      \{ x \in \R^V/\Z^V: \modone{x_i} < \modone{x_j} \text{ if }\omega'\text{ directs }i \rightarrow j\}. 
$$
\end{cor}

%%============================================================
\section{Toric extensions}
\label{extensions-section}
%%============================================================

Recall that for two (ordinary) posets $P, P'$ on a set $V$, one
says that {\it $P'$ is an extension of $P$} when $i <_P j$ implies
$i <_{P'} j$.  It is easily seen how to reformulate this geometrically:
$P'$ is an extension of $P$ if and only one has
an inclusion of their open polyhedral cones 
$c(P') \subseteq c(P)$, as defined in \eqref{cone-of-a-poset-defn}.
This motivates the following definition.

\begin{defn}
Given two toric posets $P, P'$ say that {\it $P'$ is a toric
extension of $P$} if one has an inclusion of their open chambers
$c(P') \subseteq c(P)$ within $\R^V/\Z^V$.
\end{defn}
An obvious situation where this can occur is when
one has $G=(V,E)$ and $G'=(V,E')$ two graphs on the same vertex set $V$, 
with $G$ an {\it edge-subgraph} of $G'$ in the sense that $E \subseteq E'$,

\begin{prop}
\label{edge-subgraph-chamber-refinement}
Fix $G=(V,E)$ a simple graph.

\begin{enumerate}
\item[(i)] Toric chambers in $\Chambers \AAA_{\tor}(G)$
are determined by their topological closures:
for any pair of chambers $c_1,c_2$ 
in $\Chambers \AAA_{\tor}(G)$, if 
$\bar{c}_1=\bar{c}_2$ then $c_1=c_2$.
\item[(ii)] 
If $G$ is an edge-subgraph of $G'$, then
$
\bar{c} = \bigcup_{c'} \bar{c}',
$ 
where the union runs over all toric chambers 
$c'$ in $\Chambers \AAA_{\tor}(G')$ for which
$P(c')$ is a toric extension of $P(c)$.
\end{enumerate}
\end{prop}
\begin{proof}
For (i), first note that any toric chamber $c$ in $\Chambers \AAA_{\tor}(G)$ 
has boundary $\bar{c} \setmin c$ contained in $\AAA_{\tor}(G)$.
Now assume two toric chambers $c_1, c_2$ in $\Chambers \AAA_{\tor}(G)$ have
$\bar{c}_1 = \bar{c}_2$, and we wish to show $c_1=c_2$.
Any point $x$ of $c_1$ has 
$x \in c_1 \subseteq \bar{c}_1=\bar{c}_2$.
However, $x$ cannot lie in $\AAA_{\tor}(G)$ since
$c_1$ is disjoint from $\AAA_{\tor}(G)$, so $x$ does not lie
in $\bar{c}_2 \setmin c_2 \subset \AAA_{\tor}(G)$ by our first
observation.  Hence $x$ lies in $c_2$.
But then $c_1, c_2$ are connected components of
$\R^V / \Z^V \setmin \Chambers \AAA_{\tor}(G)$, sharing the point $x$,
so $c_1=c_2$.

For (ii), we first argue that
\begin{equation}
\label{closure-described-via-lift}
\bar{c}=\pi\left( \overline{\pi^{-1}(c)} \right)
\end{equation}
using the fact that the covering map $\R^V \overset{\pi}{\rightarrow} \R^V/\Z^V$
is locally a homeomorphism.  For any point $x$ in $\R^V/\Z^V$ there is
an open neighborhood $U$ which lifts to an open neighborhood $\hat{U}$,
mapping homeomorphically under $\pi$ to $U$.  Hence $x$ is the
limit of a sequence $\{x_i\}_{i=1}^\infty$ of points in $c$ 
if and only if its lift $\hat{x}=\pi|_{\hat{U}}^{-1}(x)$ is a limit of 
the sequence of points $\{\pi|_{\hat{U}}^{-1}(x_i)\}_{i=1}^\infty$ in  
$\pi^{-1}(c)$.  This shows \eqref{closure-described-via-lift}.

Since a toric chamber $c$ has $\pi^{-1}(c)$ given by a union of
affine chambers $\hat{c}$ in $\Chambers \AAA_{\aff}(G)$, in light of
\eqref{closure-described-via-lift},
it suffices to show that any affine chamber $\hat{c}$ in $\Chambers \AAA_{\aff}(G)$ has closure $\overline{\hat{c}}$ 
given by the union of the closures
$\overline{\hat{c}'}$ taken over all affine chambers $\hat{c}'$ in $\Chambers \AAA_{\aff}(G')$ that satisfy $\hat{c}' \subseteq \hat{c}$.  
However, this
is clear since $\hat{c}$ is a polyhedron bounded by hyperplanes taken from $\AAA_{\aff}(G)$, while $\AAA_{\aff}(G')$ simply refines this decomposition with more hyperplanes.
\end{proof}

%%============================================================
\section{Toric directed paths}
\label{directed-paths-section}
%%============================================================

A particular special case of Proposition~\ref{edge-subgraph-chamber-refinement} is worth noting:
every graph $G=(V,E)$ is an edge-subgraph of the {\it complete graph $K_V$}.
As noted in the Introduction, acyclic orientations $\omega$ of $K_V$ correspond to total orders $w_1 < \cdots < w_n$, indexed by permutations $w=(w_1,\ldots,w_n)$ of $V=[n]:=\{1,2,\ldots,n\}$.  It is easy to characterize the equivalence relation $\equiv$ on these total orders, and hence the toric chambers $\Chambers \AAA_{\tor}(K_V)$, in terms of cyclic shifts of these linear orders.  However,
it is worthwhile to define this concept is a bit more generally-- it turns out to be crucial in Section~\ref{chains-section}.

\begin{defn}
\label{toric-directed-path-defn}
Given a simple graph $G=(V,E)$ and an acyclic orientation $\omega$ of
$G$, say that a sequence $(i_1,i_2,\ldots,i_m)$ of elements of $V$ forms
a {\it toric directed path in $\omega$} if $(G,\omega)$ contains all of these
edges:
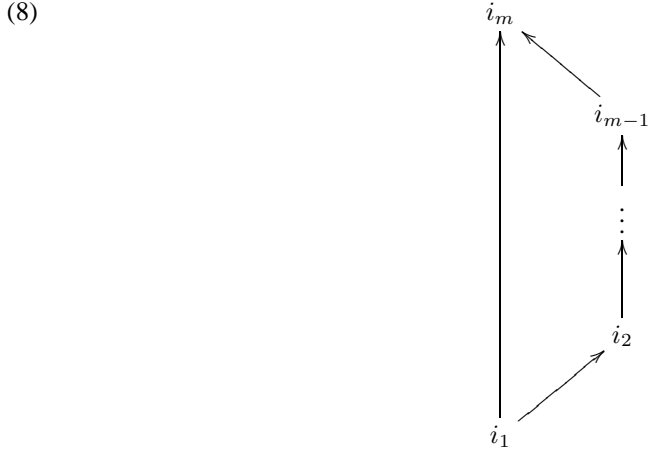
\begin{equation}
\label{toric-directed-path-figure}
\xymatrix{
i_m & \\
                     &i_{m-1} \ar[ul] \\
                     &\vdots \ar[u]\\
                     &i_2 \ar[u] \\
i_1 \ar[uuuu] \ar[ur]&
}
\end{equation}
\end{defn}
\noindent
In particular, for small values of $m$, a toric 
directed path in $\omega$
\begin{enumerate}
\item[$\bullet$]
of size $m=2$ is a directed edge $(i_1,i_2)$, 
\item[$\bullet$]
of size $m=1$ is a degenerate path $(i_1)$ for any $i_1$ in $V$, and
\item[$\bullet$]
of size $m=0$ is the empty subset $\varnothing \subset V$.
\end{enumerate}

\begin{prop}
\label{flip-preserves-toric-directed-path}
An acyclic orientation $\omega$ of $G$
contains a toric directed path $(i_1,i_2,\ldots,i_m)$ if and only if every 
acyclic orientation $\omega'$ in its $\equiv$-equivalence class contains a (unique) toric directed path 
$$
(i_\ell,i_{\ell+1},\ldots,i_m,i_1,i_2,\ldots,i_{\ell-1})
$$
which is one of its cyclic shifts, that is, it lies in the cyclic
equivalence class $[(i_1,\ldots,i_m)]$.
\end{prop}
\begin{proof}
A toric directed path $(i_1,i_2,\ldots,i_m)$ has only one source, namely $i_1$,
and only one sink, namely $i_m$.  The assertion follows by 
checking that the effect of a source-to-sink flip at $i_1$ (resp. $i_m$)
is a cyclic shift to the toric directed path $(i_2,\ldots,i_m,i_1)$ (resp.
$(i_m,i_1,i_2,\ldots,i_{m-1})$).
\end{proof}

\begin{rem}
\label{Coleman-remark}
We point out a reformulation of the sink-to-source equivalence
relation $\equiv$ on $\Acyc(G)$, due to Pretzel \cite{Pretzel:86},
leading to a reformulation of toric directed paths, useful
in Section~\ref{Antichain-section} on toric antichains.

Given a simple graph $G=(V,E)$, say that a cyclic equivalence
class $I=[(i_1,\ldots,i_m)]$ of ordered vertices is a {\it directed
cycle} of $G$ if $m\geq 3$ and $G$ contains all of the (undirected) 
edges $\{i_j,i_{j+1}\}_{j=1,2\ldots,m}$,
with subscripts taken modulo $m$.  Given such a directed cycle $I$
define \emph{Coleman's $\nu$-function}~\cite{Coleman:89}
$$
\Acyc(G) \overset{\nu_I}{\longrightarrow} \Z
$$
where $\nu_I(\omega)$ for an acyclic orientation $\omega$ of $G$ is
defined to be the number of edges $\{i_j,i_{j+1}\}$ in $I$ 
which $\omega$ orients $i_j  \rightarrow i_{j+1}$
minus the number of edges $\{i_j,i_{j+1}\}$
which $\omega$ orients $i_{j+1}  \rightarrow i_j$.
It is easy to see that $\nu_I$ is preserved
by flips, and thus extends in a well-defined manner 
to $\equiv$-classes $[\omega]$. In fact, Pretzel~\cite{Pretzel:86} 
showed that this is a complete $\equiv$-invariant:

\begin{prop}\label{prop-pretzel}
  Fixing the graph $G=(V,E)$, two acyclic orientations $\omega, \omega'$ in $\Acyc(G)$
  have $\omega \equiv \omega'$ if and only if 
  $\nu_I(\omega)=\nu_I(\omega')$ for every directed cycle $I$ of $G$.
\end{prop}
Toric directed paths then have an obvious 
characterization in terms of their $\nu_I$ function.

\begin{cor}
  \label{cor:nu}
  Given a directed cycle in $I=[(i_1,\ldots,i_m)]$ in $G$, an acyclic orientation $\omega$ in $\Acyc(G)$ contains a toric directed path lying in the
cyclic equivalence class $I$ if and only if $\nu_I(\omega)=m-1$.
\end{cor}

\end{rem}

%%============================================================
\section{Toric total orders}
\label{total-orders-section}
%%============================================================

An important special case of toric directed paths occurs when one
considers acyclic orientations of the complete graph $K_V$.  Acyclic
orientations of $K_V$ correspond to permutations $w=(w_1,\ldots,w_n)$
of $V$ (or \emph{total orders}), and always form toric directed paths
in $w$.  Hence their toric equivalence classes are the equivalence
classes $[w]$ of permutations up to cyclic shifts, or {\it toric total
  orders}. This concept coincides with the pre-existing concept of
{\it total cyclic order} from
Definition~\ref{total-cyclic-order-defn}, even though toric {\it partial}
orders are not the same as {\it partial} cyclic orders. Therefore, we
can use these terms interchangeably.

By Theorem~\ref{geometric-interpretation-theorem}, these toric total 
orders $[w]$ index the chambers $c_{[w]}$ in $\Chambers \AAA_{\tor}(K_V)$.
By Corollary~\ref{toric-chamber-as-union-cor}, one has this more concrete
description of such chambers:
\begin{equation}
\label{more-concrete-total-cyclic-chamber-description}
c_{[w]}= \bigcup_{ i=1 }^n
      \{ x \in \R^V/\Z^V: 
         \modone{x_{w_i}} < \cdots < \modone{x_{w_n}}
         < \modone{x_{w_1}} < \cdots < \modone{x_{w_{i-1}}} \}.
\end{equation}

\begin{defn}
\label{total-cyclic-extension-defn}
Given a toric poset $P=P(c)$ on $V$,
say that a toric total order $[w]$ on $V$ is a {\it toric total extension} of $P$ if the toric chamber $c_{[w]}$ of $\Chambers \AAA_{\tor}(K_V)$ is contained in $c$.  Denote by $\LLL_{\tor}(P)$ the
set of all such toric total extensions $[w]$ of $P$.
\end{defn}

The following corollary is then a special case of Proposition~\ref{edge-subgraph-chamber-refinement}.

\begin{cor}
\label{determination-by-total-cyclic-extensions}
Fix a simple graph $G=(V,E)$.
Then any toric chamber/poset $c=c(P)$ in $\Chambers \AAA_{\tor}(G)$
has topological closure
$$
\bar{c} = \bigcup_{[w] \in \LLL_{\tor}(P)} \bar{c}_{[w]}.
$$
and is completely determined by its set $\LLL_{\tor}(P)$
of toric total extensions:  
if $c_1,c_2$ in $\Chambers \AAA_{\tor}(G)$ have 
$\LLL_{\tor}(P(c_1))=\LLL_{\tor}(P(c_2))$, then $c_1=c_2$.
\end{cor}

\begin{ex}
\label{lack-of-determination-by-extensions-example}
Corollary~\ref{determination-by-total-cyclic-extensions}
fails when one does {\it not} fix the graph $G$.  For example,
when $V=\{1,2,3\}$, all $7$ of the {\it non-complete} graphs 
$G \neq K_V=K_3$
share the property that $\Chambers \AAA_{\tor}(G)$ has only one
chamber $c=c(P)$ with $\LLL_{\tor}(P)=\{ [(1,2,3)], [(1,3,2)]\}$,
whose closure $\bar{c}$ is the entire torus $\R^3/\Z^3$.
However, the unique toric chambers $c$ for these $7$ graphs 
are all different, when considered as {\it open} subsets of $\R^3/\Z^3$, 
and therefore each represents a {\it different} toric poset $P=P(c)$.

On the other hand, the complete graph $V_V=K_3$ has $2$ different 
toric equivalence classes of acyclic orientations, representing
two different chambers within the same toric arrangement $\AAA_{\tor}(K_3)$,
and two different toric posets:  $P(c_{[(1,2,3)]})$ and $P(c_{[(1,3,2)]})$.
\end{ex}

%%============================================================
\section{Toric chains}
\label{chains-section}
%%============================================================

We introduce the toric analogue of a chain (= totally ordered subset)
in a poset, and explicate its relation to the toric directed paths
from Definition~\ref{toric-directed-path-defn} and the toric total
extensions from Definition~\ref{total-cyclic-extension-defn} (or
equivalently, total cyclic extensions). 

As motivation, note that in an ordinary poset $P(c)$, a
chain $C=\{i_1,\ldots,i_m\} \subseteq V$ has the following geometric
description: there is a total ordering $(i_1,\ldots,i_m)$ of $C$ such
that every point $x$ in the open polyhedral cone $c=c(P)$ has
$x_{i_1}<x_{i_2}<\cdots<x_{i_m}$.

\begin{defn}\label{toric-chain-defn}
  Fix a toric poset $P=P(c)$ on a finite set $V$.
  Call a subset $C=\{i_1,\ldots,i_m\} \subseteq V$ 
  a {\it toric chain} in $P$ if 
  there exists a cyclic equivalence class $[(i_1,\ldots,i_m)]$ of 
  linear orderings of $C$
  with the following property:
  for every $x$ in the open toric chamber $c=c(P)$ there exists some $(j_1,\dots,j_m)$ in $[(i_1,\ldots,i_m)]$ for which
  \begin{equation}
    \label{eq:toric-chain}
    \modone{x_{j_1}}<\modone{x_{j_2}}<\cdots<\modone{x_{j_m}}.
  \end{equation}
In this situation, we will say that $P|_C=[(i_1,\ldots,i_m)]$. 
\end{defn}

\begin{rem}
\label{toric-chain-defn-remark}
Note that 
\begin{enumerate}
\item[$\bullet$]
singleton sets $\{i\}$ and the empty subset $\varnothing \subset V$ are
always toric chains in $P$,
\item[$\bullet$]
subsets of toric chains are toric chains, and
\item[$\bullet$]
a pair $\{i,j\}$ is a toric chain in $P=P(c)$ if and only 
if every point $x$ in the open toric chamber
$c$ has $\modone{x_i} \neq \modone{x_j}$;
in particular, this will be true whenever $c$ as appears as 
toric chamber in $\Chambers \AAA_{\tor}(G)$ for a graph $G$ having
$\{i,j\}$ as an edge of $G$.
\end{enumerate}
\end{rem}

Though the definition of toric chain does not refer to a particular
graph $G$, there are several convenient characterizations that involve
a graph. In the following proposition, we list five equivalent
conditions. 
The exception when $|C|\neq 2$ is needed because the last
condition is vacuously true whenever $|C|=2$; in this case, only the
first four are equivalent.

\begin{prop}\label{toric-chain-prop}
  Fix a toric poset $P=P(c)$ on a finite set $V$, and 
  $C=\{i_1,\ldots,i_m\} \subseteq V$.
  The first four of the following five conditions are equivalent,
  and when $m=|C| \neq 2$, they are also equivalent to the fifth.
\begin{enumerate}
  
\item[(a)] $C$ is a toric chain in $P$, with $P|_C=[(i_1,\ldots,i_m)]$.
  
\item[(b)] For every graph $G=(V,E)$ and acyclic orientation $\omega$
  of $G$ having $\bar\alpha_G(c)=[\omega]$, the subset $C$ is a chain
  in the poset $P(G,\omega)$, ordered in
  some cyclic shift of the order $(i_i,\dots,i_m)$.

\item[(c)] For every graph $G=(V,E)$ and acyclic orientation $\omega$
  of $G$ having $\bar\alpha_G(c)=[\omega]$, the subset $C$ occurs as a
  subsequence of a toric directed path in $\omega$, in some cyclic
  shift of the order $(i_i,\dots,i_m)$.
  
\item[(d)] There exists a graph $G=(V,E)$ and acyclic orientation
  $\omega$ of $G$ having $\bar\alpha_G(c)=[\omega]$ such that $C$ occurs as a
  subsequence of a toric directed path in $\omega$, in some cyclic shift
  of the order $(i_1,\dots,i_m)$.

\item[(e)] Every total cyclic extension $[w]$ in $\LLL_{\tor}(P(c))$ has the
same restriction $[w|_C]=[(i_1,\ldots,i_m)]$.
\end{enumerate}
\end{prop}

\noindent
The following easy and well-known lemma will be used in the proof.

\begin{lem}
\label{incomparability-lemma}
When two elements $i,j$ are incomparable in a finite poset $Q$ on $V$,
one can choose a linear extension $w=(w_1,\ldots,w_n)$
in $\LLL(Q)$ that has $i,j$ appearing consecutively,
say $(w_s,w_{s+1})=(i,j)$.
\end{lem}
\begin{proof}
Begin $w$ with any linear extension 
$w_1,w_2,\ldots,w_{s-1}$
for the order ideal $Q_{<i} \cup Q_{<j}$,
followed by $w_s=i, w_{s+1}=j$, and
finish with any linear extension
$w_{s+2},w_{s+3},\ldots,w_n$
for $Q \setmin \left(Q_{\leq i} \cup Q_{\leq j}\right)$.
\end{proof}

\begin{proof}[Proof of Proposition~\ref{toric-chain-prop}]
Note that if $|C|\leq 1$, all five conditions (a)-(e) are vacuously true,
so without loss of generality $|C| \geq 2$.
We will first show (a) implies (b) implies (c) implies (d) implies (e).
Then we will show that (e)
implies (a) when $|C|\geq 3$, and (d) implies (a) when $|C|=2$. 

\vskip.1in
\noindent {\sf (a) implies (b).} Assume that $C$ is a toric chain of $P$,
with $P|_C=[(i_1,\ldots,i_m)]$, 
and take a graph $G$ and orientation $\omega$ such that
$\bar{\alpha}_G(c)=[\omega]$. 

We first show by contradiction that $C$ must be totally ordered in
$Q:=P(G,\omega)$.  Assume not, and 
say $i,j$ in $C$ are incomparable in $Q$. By Lemma~\ref{incomparability-lemma}
there is a linear extension $w=(w_1,\ldots,w_n)$ in $\LLL(Q)$ 
having $i,j$ appear consecutively, say $(w_s,w_{s+1})=(i,j)$.
Choose $x$ in $\R^n$ with 
$0\leq x_{w_1}<\cdots<x_{w_n}<1$
and let $x'$ be obtained by $x$ by exchanging $x_i,x_j$,
that is $x'_i=x_j$ and $x'_j=x_i$.
Since $x=\modone{x}$ and $x'=\modone{x'}$, 
one has $\bar{\alpha}_G(x)=\omega=\bar{\alpha}_G(x')$,
and hence $x, x'$ lie in $c=c(P)$.
The condition \eqref{eq:toric-chain} on $x, x'$ implies that
$[w|_C]=[w'|_C]$ should give the same cyclic order on $C$,
which forces $m=2$ and $C=\{i,j\}$.  However, the average
$x''=\frac{x+x'}{2}$ gives a third point in $c$
having $\modone{x''_i}=x''_i=x''_j=\modone{x''_j}$, 
contradicting \eqref{eq:toric-chain}.

Once one knows that $C$ is totally ordered in $Q$,
consideration of \eqref{eq:toric-chain} for the point $x$
chosen as above implies that $w|_C$ lies in $[(i_1,\ldots,i_m)]$,
and hence the same is true of $Q|_C$.

\vskip.1in
\noindent {\sf (b) implies (c).}
Assume for the toric poset $P=P(c)$,
every graph $G$ and orientation $\omega$ with
$\bar{\alpha}_G(c)=[\omega]$ 
has $C$ totally ordered in $P(G,\omega)$ by
a cyclic shift $(j_1,\ldots,j_m)$ in $[(i_1,\ldots,i_m)]$.
We will show that $C$ actually occurs in this order 
as a subsequence of some toric directed path in $\omega$.

By Proposition~\ref{flip-preserves-toric-directed-path}, one is free to
alter $\omega$ within the class $[\omega]$.
So choose $\omega$ within $[\omega]$  among
all those for which $P(G,\omega)$ on $V$ totally orders $C$ as
$j_1 < \cdots < j_m$, but minimizing the cardinality $|Z|$
where
$$
Z:=\{z \in V: z \text{ there is a directed }\omega\text{ path from }j_m\text{ to }z\}
$$
Note that $Z$ is nonempty, since it contains $j_m$.
We claim that minimality forces $|Z|=1$, that is, $Z=\{j_m\}$.  
To argue the claim by contradiction, assume $Z \neq \{j_m\}$.
Then one can find an $\omega$-sink $z \neq j_m$ in $Z$, as $V$ is finite,
and $\omega$ is acyclic.
Perform a sink-to-source flip at $z$ to create a new orientation
$\omega'$ in $[\omega]$.  Then $\omega'$ 
still has $P(G,\omega')$ totally ordering $C$ as
$j_1 < \cdots < j_m$, but its set $Z'$ has $|Z'|<|Z|$ because
$Z' \subset Z \setmin \{z\}$.

Now $Z=\{j_m\}$ means that $j_m$ is an $\omega$-sink.  Create
$\omega'$ by flipping $j_m$ from sink to source.
Since $j_1$ is supposed to be comparable with 
$j_m$ in $P(G,\omega')$, one must have
$j_m <_{P(G,\omega')} j_1$, that is, there is an $\omega'$-path 
of the form 
$j_m\rightarrow k \rightarrow \cdots \rightarrow j_1$;
possibly $k=j_1$ here.
But this means that prior to the sink-to-source flip of $j_m$,
one had a toric directed $\omega$-path 
$k \rightarrow \cdots \rightarrow j_1 \rightarrow j_2 \rightarrow \cdots \rightarrow j_m$
that contained $C$, as desired.

\vskip.1in
\noindent {\sf (c) implies (d).}
Trivial.

\vskip.1in
\noindent
{\sf (d) implies (e).}
Assume the graph $G$ has $\bar{\alpha}_G(c)=[\omega]$ and
$C$ occurs in the order $(i_1,\ldots,i_m)$ as a 
subsequence of a toric directed path in $\omega$.  
We must show that every total cyclic extension $[w]$ of $P=P(c)$
has restriction $[w|_C] = [(i_1,\ldots,i_m)]$.

By Definition~\ref{total-cyclic-extension-defn}, one has $c_{[w]} \subseteq c$.  By \eqref{more-concrete-total-cyclic-chamber-description}, 
one can pick a point $x$ in $c_{[w]}$, so that 
$$
\modone{x_{w_1}} < \cdots < \modone{x_{w_n}}.
$$  
Since $x$ also lies in $c$, one has $\bar{\alpha}_G(x) = \omega' \equiv \omega$.
Proposition~\ref{flip-preserves-toric-directed-path} implies
that $\omega'$ contains as a toric directed path
some cyclic shift $(j_1,\ldots,j_m)$ of $(i_1,\ldots,i_m)$.
Hence 
$$
\modone{x_{j_1}} < \cdots < \modone{x_{j_m}},
$$
which forces $w|_C=(j_1,\ldots,j_m)$, as desired.

\vskip.1in
\noindent
{\sf (e) implies (a) when $|C|\geq 3$.}
Assume that every total cyclic extension $[w]$ of $P=P(c)$ has $w|_C$
lying in the same cyclic equivalence class $[(i_1,\ldots,i_m)]$.
We want to show that every point $x$ in $c$
satisfies \eqref{eq:toric-chain}.
Recall from Corollary~\ref{toric-chamber-as-union-cor}
that there is at least one graph $G$
and $\equiv$-class $[\omega]$ containing
$\bar{\alpha}_G(x)$, that is, $\bar{\alpha}_G(c)=[\omega]$.
It suffices to show that 
the partial order $Q:=P(G,\omega)$ on $V$ induced by any orientation
$\omega$ in this $\equiv$-class has restriction $Q|_C$
to the subset $C$ giving a total order $(j_1,\ldots,j_m)$,
and this total order lies in $[(i_1,\ldots,i_m)]$.

Suppose that $Q|_C$ is {\it not} a total order;
say elements $i,j$ in $C$ are incomparable in $Q$.  By Lemma~\ref{incomparability-lemma}, one can then
choose linear extensions $w,w'$ in
$\LLL(Q)$ that both have $i,j$ consecutive, and differ only
in swapping $i,j$, say $(w_s,w_{s+1})=(i,j)$ and $(w'_s,w'_{s+1})=(j,i)$.
Pick points $x,x'$ that satisfy
\begin{align*}
&0 \leq x_{w_1} < \cdots < x_{w_n} < 1\\
&0 \leq x'_{w'_1} < \cdots < x'_{w'_n} < 1.
\end{align*}
Since $x=\modone{x}, x'=\modone{x'}$,
one finds that $x,x'$ lie in $c_{[w]}, c_{[w']}$, respectively.
Also one has $\bar{\alpha}_G(x)=\omega=\bar{\alpha}_G(x')$ so that 
both $x, x'$ lie in $c$.  Hence $c_{[w]},c_{[w']} \subseteq c$, that
is, both $[w],[w']$ are total cyclic extensions in $\LLL_{\tor}(P(c))$.  
However, since $|C| \geq 3$, there exists some third 
element $k$ in $C \setmin \{i, j\}$, and $[w],[w']$ differ
in their cyclic ordering of $\{i,j,k\}$.  This contradicts
assumption (e), so $Q|_C$ is a total order.  

Once one knows that $Q|_C$ is a total order $j_1 < \cdots <j_m$,
the above argument shows that $(j_1,\ldots,j_m)$ lies in the cyclic
equivalence class $[w|_C]$ for every $w$ in $\LLL_{\tor}(P)$, which
is $[(i_1,\ldots,i_m)]$ by assumption.

\vskip.1in
\noindent
{\sf (d) implies (a) when $|C|=2$.}  Suppose
$\bar\alpha_G(c)=[\omega]$ and $C$ occurs as a subsequence of a toric
directed path in $\omega$, with $i_1<i_2$. By
Proposition~\ref{flip-preserves-toric-directed-path}, if
$\omega'\equiv\omega$, then $C$ occurs in a toric directed path in
$\omega'$. This means that for any $x$ with $\bar{\alpha}_G(x)=\omega'$, 
we have $\modone{x_{i_1}}\neq\modone{x_{i_2}}$, and
so either $\modone{x_{i_1}}<\modone{x_{i_2}}$ or
$\modone{x_{i_2}}<\modone{x_{i_1}}$ must hold for every $x$ in
$c$. Thus $C$ is a toric chain of $P(c)$.
\end{proof}

%%============================================================
\section{Toric transitivity}
\label{transitivity-section}
%%============================================================

We next clarify the edges that are ``forced'' in a toric partial order,
an analogue of transitivity that we refer as \emph{toric transitivity}. 

\begin{thm}
\label{thm:toric-transitivity}
  Fix a toric poset $P=P(c)$ on $V$, and assume that $G=(V,E)$ has
$c$ appearing as a toric chamber in $\Chambers \AAA_{\tor}(G)$,
say $\bar{\alpha}_G(c)=[\omega]$.  Then for any 
non-edge pair $\{i,j\} \not\in E$, either
\begin{enumerate}
\item[(i)] $i,j$ lies on a toric chain in $P$, in which case
$c$ is also a toric chamber for $G^+=(V,E\cup \{i,j\})$, 
and there is a unique extension $\omega^+$ of $\omega$ 
such that $\bar{\alpha}_{G^+}(c)=\omega^+$, or
\item[(ii)] $i,j$ lies on no toric chains in $P$, and then
the hyperplane $\modone{x_i = x_j}$ intersects the open toric chamber $c$.
\end{enumerate}
\end{thm}

\begin{proof}
Assertion (i) follows from Proposition~\ref{toric-chain-prop}:  when $i,j$ lie
on a toric chain $C$ in $P$, assertion (b) of that proposition says that
they lie on a toric directed path in $\omega$ 
for every representative of the class $[\omega]$, and hence
the inequality $\modone{x_i} < \modone{x_j}$ (or its reverse inequality) 
is already 
implied by the other inequalities defining the points of
$\bar{\alpha}_G^{-1}(\omega)$ that come from the edges of $G$ induced by $C$.

For assertion (ii), note that whenever there exist no points $x$ of the open toric chamber $c$ having $\modone{x_i} = \modone{x_j}$, then every $x$ in $c$ has
either $\modone{x_i} < \modone{x_j}$ or
$\modone{x_j} < \modone{x_i}$.  This shows that
$\{i,j\}$ is itself a toric chain in $P=P(c)$;  see 
Remark~\ref{toric-chain-defn-remark}.
\end{proof}

This suggests the following definition.

\begin{defn}
\label{toric-transitive-closure-defn}
Given a graph $G=(V,E)$ and $\omega$ in $\Acyc(G)$,
the {\it toric transitive closure} of the pair
$(G,\omega)$ is the pair
$
(\bar{G}^{\tor},\bar{\omega}^{\tor})
$
defined as follows.
The edges of $\bar{G}^{\tor}$ are obtained by
adding to the edges of $G$ all pairs 
$\{i,j\}$ that are a subset of some toric directed path in $\omega$;
see the dotted edges in \eqref{toric-transitivity-figure} below.
The acyclic orientation 
$\bar{\omega}^{\tor}$ orients the edge $i \rightarrow j$ if the
toric directed path contains a path from $i$ to $j$, rather than 
from $j$ to $i$.\end{defn}

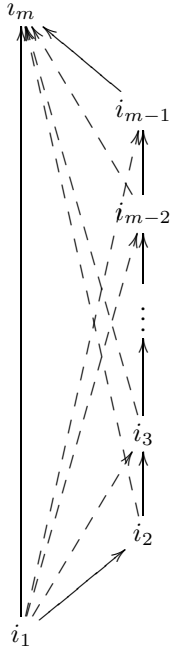
\begin{equation}
\label{toric-transitivity-figure}
\xymatrix{
i_m & \\
                     &i_{m-1} \ar[ul]\\
                     &i_{m-2} \ar[u]\ar@{-->}[uul]\\
                     &\vdots \ar[u]\\
                     &i_3 \ar[u]\ar@{-->}[uuuul]\\
                     &i_2 \ar[u]\ar@{-->}[uuuuul] \\
i_1 \ar[uuuuuu] \ar[ur]\ar@{-->}[uur]\ar@{-->}[uuuur]\ar@{-->}[uuuuur]&
}
\end{equation}

\begin{cor}
\label{toric-transitive-closure-independence}
The toric transitive closure 
depends only upon the toric poset $P=P(c)$ which satisfies $\bar{\alpha}_G(c)=[\omega]$,
in the following sense:  
given two graphs $G_i=(V,E_i)$ for $i=1,2$,
and $\omega_i$ in $\Acyc(G_i)$ with $\bar{\alpha}_{G_i}(c)=[\omega_i]$, then
\begin{enumerate}
\item[(i)] 
$\bar{G}^{\tor}_1=\bar{G}^{\tor}_2$, and
\item[(ii)]
$\bar{\omega}^{\tor}_1 \equiv \bar{\omega}^{\tor}_2$.
\end{enumerate}
\end{cor}

\begin{proof}
Assertion (i) follows from the fact that $\{i,j\}$ appears as an edge
in $\bar{G}^{\tor}$ if and only if it is a subset of some toric chain of
$P$, and adding $\{i,j\}$ does not affect the toric poset $P=P(c)$,
according to Theorem~\ref{thm:toric-transitivity}(i).  For assertion
(ii), note that iterating Theorem~\ref{thm:toric-transitivity}(i)
gives
$$
\bar{\alpha}_{\bar{G}^{\tor}}^{-1}(\bar{\omega}_1^{\tor})=
\bar{\alpha}_{G_1}^{-1}(\omega_1)=
c=\bar{\alpha}_{G_2}^{-1}(\omega_2)
=\bar{\alpha}_{\bar{G}^{\tor}}^{-1}(\bar{\omega}_2^{\tor}).
$$
Assertion (ii) then follows from Theorem~\ref{geometric-interpretation-theorem}.
\end{proof}

\begin{rem}
Note that the toric transitive closure of $\bar{A}^{\tor}$ is
always a subset of the ordinary transitive closure $\bar{A}$, since
any toric directed path that contains $(i,j)$ as a subsequence
also contains an ordinary directed path from $i$ to $j$.
\end{rem}

%%============================================================
\section{Proof of Theorem~\ref{toric-convex-closure-theorem}}
\label{convex-closure-section}
%%============================================================

Here we wish to regard a pair $(G,\omega)$ of a simple graph $G=(V,E)$
and acyclic orientation $\omega$ in $\Acyc(G)$ as a subset 
$A \subset \overleftrightarrow{K}_V$ of the set
of all possible directed edges
$\overleftrightarrow{K}_V=\{ (i,j) \in V \times V : i \neq j \}$. 
Then the toric transitive closure operation
$(G,\omega) \longmapsto (\bar{G}^{\tor},\bar{\omega}^{\tor})$
from Definition~\ref{toric-transitive-closure-defn}
may be regarded as a {\it closure operator} on $\overleftrightarrow{K}_V$,
that is, a map $A \longmapsto \bar{A}^{\tor}$
from $2^{\overleftrightarrow{K}_V}$ to itself, satisfying
\begin{itemize}
\item $A \subseteq \bar{A}^{\tor}$,
\item $A \subseteq B$ implies $\bar{A}^{\tor} \subseteq \bar{B}^{\tor}$, and
\item $\bar{\bar{A}}^{\tor}=\bar{A}^{\tor}$.
\end{itemize}

\noindent
Recall the statement of Theorem~\ref{toric-convex-closure-theorem}:

\vskip.1in
\noindent
{\bf Theorem~\ref{toric-convex-closure-theorem}.}
{\it
The toric transitive closure operation $A \longmapsto \bar{A}^{\tor}$
is a {\it convex closure}, that is,
$$
\text{ for } a \neq b \text{ with }
a,b \not\in \bar{A}^{\tor} \text{ and }
a \in \overline{A \cup \{b\}}^{\tor},
\text{ one has }b \notin \overline{A \cup \{a\}}^{\tor}.
$$
}

For the purposes of the proof, introduce one further
bit of terminology.
\begin{defn}
For $\omega$ in $\Acyc(G)$ and a toric directed path $C=(i_1,\ldots,i_m)$ in
$\omega$ of size $m\geq 3$, as in \eqref{toric-directed-path-figure},
call $(i_1,i_m)$ the {\it long edge} of $C$, and
call the other edges $(i_1,i_2),(i_2,i_3),\ldots,(i_{m-1},i_m)$ the
{\it short edges} of $C$.
\end{defn}

\begin{proof}[Proof of Theorem~\ref{toric-convex-closure-theorem}.]
Proceed by contradiction:
suppose $(i,j) \neq (k,\ell)$ are {\it not} in
$\bar{A}^{\tor}$, but both 
\begin{enumerate}
\item[$\bullet$]
$(k,\ell)$ lies in $\overline{A\cup (i,j)}^{\tor}$, 
say because $(i,j)$ creates a toric directed path $C$
also containing $(k,\ell)$, which was 
not already present in $\bar{A}^{\tor}$,
and
\item[$\bullet$]
$(i,j)$ lies in $\overline{A\cup (k,\ell)}^{\tor}$, 
say because $(k,\ell)$ creates a toric directed path $D$
also containing $(i,j)$, which was
not already present in $\bar{A}^{\tor}$.
\end{enumerate}
Introduce the (ordinary) partial order $Q$ on $V$ which
is the (ordinary) transitive closure of
$\bar{A}^{\tor} \cup \{(i,j),(k,\ell)\}$.
We use this to argue a contradiction in various cases.

\vskip.1in
\noindent
{\sf Case 1. Either $(i,j)$ is the long edge of $C$, or
$(k,\ell)$ is the long edge of $D$.}
By relabeling, assume without loss of generality
that $(i,j)$ is the long edge of $C$.
Then in $Q$, one has 
\begin{equation}
\label{first-transitivity-inequalities}
i \leq k < \ell \leq j
\end{equation}
with at least one of the two weak inequalities being strict.  

\vskip.1in
\noindent
{\it Subcase 1a. $(k,\ell)$ is also the long edge of $D$.}
Then in $Q$ one also has $k \leq i < j \leq \ell$, which  
with \eqref{first-transitivity-inequalities} gives
$$
k \leq i \leq k < \ell \leq j \leq \ell
$$
forcing the contradiction $(i,j)=(k,\ell)$.

\vskip.1in
\noindent
{\it Subcase 1b. $(k,\ell)$ is a short edge of $D$.}
Then since $C$ has $(i,j)$ as its long edge and gives a toric directed path
containing $(k,\ell)$ (while $\bar{A}^{\tor}$ had no such path),
$C$ must contain a directed path from $k$ to $\ell$ with at least two steps.
Combining this with $D \setmin \{(k,\ell)\}$ gives
a toric directed path in $\bar{A}^{\tor}$ that contains $(i,j)$; contradiction.

\vskip.1in
\noindent
{\sf Case 2. Both $(i,j),(k,\ell)$ are
short edges of $C, D$, respectively.}
In this case, $\bar{A}^{\tor}$ cannot contain a path
from $i$ to $j$, else replacing $(i,j)$ in $C$ with this path
would give the contradiction that $(i,j)$ is in $\bar{A}^{\tor}$.
Similarly, $\bar{A}^{\tor}$ cannot contain a path
from $k$ to $\ell$.  Also note that, since $C$ (or $D$) is a directed path
containing all four of $\{i,j,k,\ell\}$, the four of them are totally ordered in $Q$.
We now argue in subcases based on how $Q$ totally orders $\{i,j,k,\ell\}$.

\vskip.1in
\noindent
{\it Subcase 2a. Either $Q$ has $i < j \leq k < \ell$
or  $k < \ell \leq i < j$.}
In this case, adding $(i,j)$ to $\bar{A}^{\tor}$ cannot help to create
a directed path from $k$ to $\ell$, contradicting the existence of $C$.

\vskip.1in
\noindent
{\it Subcase 2b. Either $Q$ has $i \leq k < \ell \leq j$,
with at least one of the weak inequalities strict,
or $k \leq i < j \leq \ell$, 
with at least one of the weak inequalities strict.}
Assume without loss of generality, by relabeling, that
one is in the first case $i \leq k < \ell \leq j$.  But
then adding $(i,j)$ to $\bar{A}^{\tor}$ again cannot help to create
a directed path from $k$ to $\ell$, contradicting the existence of $C$.

\vskip.1in
\noindent
{\it Subcase 2c. Either $Q$ has $i \leq k \leq j \leq \ell$,
with at least two consecutive strict inequalities,
or $k \leq i \leq \ell \leq j$,
with at least two consecutive strict inequalities.}
Assume without loss of generality, by relabeling, that
one is in the first case $i \leq k \leq j \leq \ell$.
But then the consecutive strict inequalities either imply
the existence within $\bar{A}^{\tor}$ of 
a directed path from $i$ to $j$,
or one from $k$ to $\ell$;  contradiction.
\end{proof}

%%============================================================
\section{Toric Hasse diagrams}
\label{Hasse-diagram-section}
%%============================================================

For convex closures $A \longmapsto \bar{A}$, it is
well-known that for any subset $A$, its {\it extreme
  points}
$$
\extreme(A): = \{ a \in A: a \not\in \overline{A \setmin \{a\}} \}
$$
gives the unique set which is minimal under inclusion among all subsets
having the same closure as $A$; see \cite{Edelman:85}.  
For ordinary transitive closure of an acyclic orientation $(G,\omega)$
as a subset of $\overleftrightarrow{K_V}$, its extreme points 
are exactly the subset of directed edges $(i,j)$ in the usual {\it Hasse diagram} for its associated partial order $P$. This suggests the following definition.

\begin{defn}
\label{toric-Hasse-diagram-defn}
Given a graph $G=(V,E)$ and $\omega$ in $\Acyc(G)$, 
corresponding to a subset $A$ of $\overleftrightarrow{K_V}$,
its {\it toric Hasse diagram} is the pair
$(\hat{G}^{\torHasse},\omega^{\torHasse})$ corresponding
to its subset of extreme points $\extreme(A)$ with respect to
the toric transitive closure operation $A \longmapsto \bar{A}^{\tor}$. The toric Hasse diagram of a toric poset $P$ is $(G^{})$
\end{defn}

Definition~\ref{toric-transitive-closure-defn} allows one to rephrase this as
follows:  
\begin{itemize}
\item $\hat{G}^{\torHasse}$ is obtained from $G$ by 
removing all {\it chord edges} $\{i_j,i_k\}$ with $|j-k| \geq 2$
from all toric directed paths $C=\{i_1,\ldots,i_m\}$ in $\omega$ 
that have $m=|C| \geq 4$, and 
\item $\omega^{\torHasse}$ is the restriction $\omega|_{\hat{G}^{\torHasse}}$.  
\end{itemize}

One then has the following analogue of 
Corollary~\ref{toric-transitive-closure-independence}.

\begin{cor}
\label{toric-Hasse-diagram-closure-independence}
The toric Hasse diagram 
depends only on the toric poset $P=P(c)$ having $\bar{\alpha}_G(c)=[\omega]$,
in the following sense:  
given two graphs $G_i=(V,E_i)$ for $i=1,2$,
and $\omega_i$ in $\Acyc(G_i)$ with $\bar{\alpha}_G(c)=[\omega_i]$, then
\begin{enumerate}
\item[(i)] 
$\hat{G}_1^{\torHasse}=\hat{G}_2^{\torHasse}$, and
\item[(ii)]
$\omega_1^{\torHasse} \equiv \omega_2^{\torHasse}$.
\end{enumerate}
\end{cor}

\begin{proof}
Same as the proof of Corollary~\ref{toric-transitive-closure-independence}.
The key point is that the toric directed paths 
$C=\{i_1,\ldots,i_m\}$ in $\omega$ are the toric chains in $P$,
and when $|C| \geq 4$, removing chords from $C$ still keeps it a toric chain.
\end{proof}

%%============================================================
\section{Toric antichains}
\label{Antichain-section}
%%============================================================

Since chains in posets have a good toric analogue, one might ask if the same is true for antichains.
Recall that an {\it antichain} of an 
ordinary poset $P$ on $V$ is a subset $A=\{i_1,\dots,i_m\} \subseteq V$ 
characterized
\begin{enumerate}
\item[$\bullet$] {\it combinatorially} by the condition
that no pair $\{i,j\}\subset A$ with $i \neq j$ are comparable, 
that is, they lie on no chain of $P$, or 
\item[$\bullet$] {\it geometrically} by the equivalent condition 
that the $(|V|-m+1)$-dimensional linear subspace
$
\{x\in\R^V: x_{i_1}=x_{i_2}=\cdots=x_{i_m}\}
$
intersects the open polyhedral cone/chamber $c(P)$ in $\R^V$. 
\end{enumerate}
In the toric situation, these two conditions
lead to different notions of toric antichains.

\begin{defn}
\label{toric-antichain-defn}
Given a toric poset $P=P(c)$ on the finite set $V$,
say that $A=\{i_1,\dots,i_m\} \subseteq V$
is a
\begin{enumerate}
\item[$\bullet$] {\it combinatorial toric antichain} of $P$ if no 
$\{i,j\}\subset A$ with $i \neq j$ lie on a common toric chain of $P$.
\item[$\bullet$] {\it geometric toric antichain} if the 
subspace $\{x\in\R^V/\Z^V:  x_{i_1}=x_{i_2}=\cdots=x_{i_m}\}$ 
intersects the open toric chamber $c=c(P)$. 
\end{enumerate}
\noindent
By analogy to the notion of the width of a poset, which is the size of its
largest antichain, define the \emph{geometric (resp. combinatorial) toric width} of a toric poset to be the size of the largest geometric (resp. combinatorial) toric antichain. 
\end{defn}

Given a toric poset $P=P(c)$ and a graph $G=(V,E)$ with 
$\bar{\alpha}_G(c)=[\omega]$, the definition and
Corollary~\ref{toric-chamber-as-union-cor} imply
that $A\subseteq V$ is a geometric toric antichain of $P$ 
if and only if $A$ is an antichain of $P(G,\omega')$ 
for some $\omega'\equiv\omega$. 
The following proposition should also be clear.

\begin{prop}
In a toric poset $P$, every geometric toric antichain is a
combinatorial toric antichain.  Thus its geometric toric width is 
bounded above by its combinatorial toric width.
\end{prop}

The next example shows that the inequality between
these two notions of toric width can be strict.

\begin{ex}
  Consider the toric poset $P=P(c)$ whose toric Hasse diagram is the
  circular graph $G=C_6$ and for which $\bar{\alpha}_G(c)$ contains
  the following representatives $\omega_1$, $\omega_2$ and $\omega_3$
  of $\Acyc(G)$:
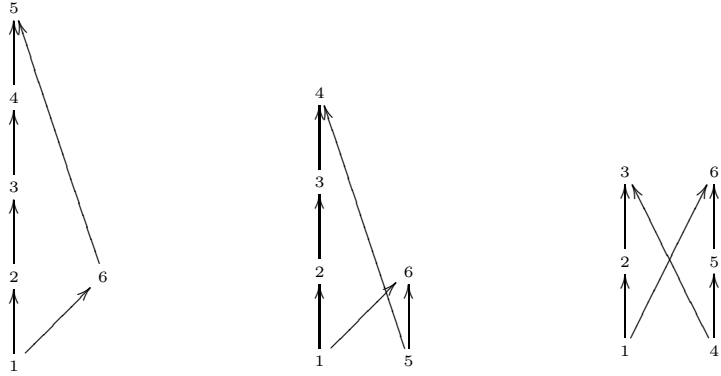
  \[
  \tiny{
    \xymatrix{
      5                & \\
      4 \ar[u]         & \\
      3 \ar[u]         & \\
      2 \ar[u]         & 6 \ar[uuul] \\
      1 \ar[u] \ar[ur] & }\hspace{1in}
    \xymatrix{
                       & \\
      4                & \\
      3 \ar[u]         & \\
      2 \ar[u]         & 6 \\
      1 \ar[u] \ar[ur] & 5 \ar[u]\ar[uuul] }\hspace{1in}
    \xymatrix{
                        & \\
                        & \\
      3                 & 6 \\
      2 \ar[u]          & 5 \ar[u] \\
      1 \ar[u] \ar[uur] & 4 \ar[u]\ar[uul] }}
  \]
  All three of these orientations satisfy $\nu_I(\omega_i)=2$ for the
  directed cycle $I=[(1,2,3,4,5,6)]$ of $G$, where $\nu_I$ is
  Coleman's $\nu$-function from Remark~\ref{Coleman-remark}.
  Moreover, Proposition~\ref{prop-pretzel}
  says that $\nu_I(\omega)=2$ must hold for any other $\omega$ in
  $[\omega_i]$. It is easy to check that for any such $\omega$, the
  directed graph $(G,\omega)$ must be isomorphic to either
  $(G,\omega_1)$, $(G,\omega_2)$, or $(G,\omega_3)$. 

  Consequently, $P$ has no toric chains except for those of cardinality $0,1,2$,
  that is, the empty set $\varnothing$, the $6$ singletons 
  and the $6$ edge pairs in $G$.  From this one can easily 
  check that the combinatorial toric antichains of $P$ are the 
  empty set $\varnothing$, the $6$ singletons, the pairs $\{i,j\}$
  which do not form edges of $G$, and the two triples $\{1,3,5\}, \{2,4,6\}$.
  In particular, $P$ has combinatorial toric width $3$.

  However, we claim neither of these triples $\{1,3,5\},\{2,4,6\}$ can
  be a geometric toric antichain, so that the geometric toric width of 
  $P$ is $2$.  To argue that $\{1,3,5\}$ is not a geometric toric antichain,
  consider three paths of length $2$ in $G$ between the
  elements of $\{1,3,5\}$, that is, the paths 
  $$
  \begin{aligned}&1-2-3\\&3-4-5\\&5-6-1\end{aligned}
  $$ 
  The only way one could avoid having an $\omega$-directed path between two
  elements of $\{1,3,5\}$ would be if $\omega$ orients both edges in each 
  of the three paths listed above
  in opposite directions. But this would lead to $\nu_I(\omega)=0$
  which is impossible for $\omega$ in $[\omega_i]$.  The argument for
  $\{2,4,6\}$ is similar.
\end{ex}

Despite the difference in the two notions of toric width,
one might still hope that one of the notions gives
a toric analogue for one or both of these two
classic results on chains and antichains in ordinary posets.

\begin{thm}\label{thm:Dilworth}
  For any (ordinary) finite poset $P$, one has:
\begin{enumerate}
\item[(i)] Dilworth's Theorem \cite{Dilworth:50}:
$$
    \max \{ |A|: A \text{ an antichain in }P \}=
    \min \{ \ell: V=\cup_{i=1}^\ell C_i, \text{ with }C_i
    \text{ chains in }P \}
$$
\item[(ii)] Mirsky's Theorem \cite{Mirsky:71}:
$$
    \max \{ |C|: C \text{ a chain in }P \}=
    \min \{ \ell: V=\cup_{i=1}^\ell A_i, \text{ with }A_i
    \text{ antichains in }P \}.
$$
\end{enumerate}
\end{thm}

\noindent
One at least has the following
inequalities, coming from the easy observation
that a toric chain and toric antichain (whether combinatorial or geometric)
can intersect in at most one element.

\begin{prop}\label{prop-toric-dilworth}
  For a toric poset $P$, both versions (geometric or
  combinatorial) of a toric antichain lead to the following inequalities
  holds:
  \[
  \begin{array}{lcl}
    \max \{ |A|: A \text{ a toric antichain in }P \}&\leq&
    \min \{ \ell: V=\cup_{i=1}^\ell C_i, \text{ with }C_i
    \text{ toric chains in }P \} \\
    \max \{ |C|: C \text{ a toric chain in }P \}&\leq&
    \min \{ \ell: V=\cup_{i=1}^\ell A_i, \text{ with }A_i
    \text{ toric antichains in }P \}. 
  \end{array}
  \]
\end{prop}

\noindent
However, the following example shows that both inequalities in
Proposition~\ref{prop-toric-dilworth} can be strict:  neither of
our two notions of toric antichain leads to a version
of Dilworth's Theorem, nor of Mirsky's theorem.

\begin{ex}\label{ex:mirsky-counterexample}
  Consider the toric poset $P=P(c)$ whose toric Hasse diagram is the
  circular graph $G=C_5$ and for which $\bar{\alpha}_G(c)$ contains
  the following representatives $\omega_1$ and $\omega_2$ of
  $\Acyc(G)$:
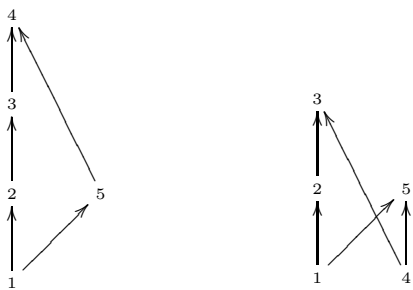
  \[
  \tiny{
    \xymatrix{
      4                & \\
      3 \ar[u]         & \\
      2 \ar[u]         & 5 \ar[uul] \\
      1 \ar[u] \ar[ur] & }\hspace{1in}
    \xymatrix{
                       & \\
      3                & \\
      2 \ar[u]         & 5 \\
      1 \ar[u] \ar[ur] & 4 \ar[u]\ar[uul] }}
  \]
  Both orientations above satisfy $\nu_I(\omega_i)=1$ for the directed
  cycle $I=[(1,2,3,4,5)]$ of $G$. Proposition~\ref{prop-pretzel} says
  that $\nu_I(\omega)=1$ must hold for any other $\omega$ in
  $[\omega_i]$, and so for such an $\omega$, the directed graph
  $(G,\omega)$ must be isomorphic to either $(G,\omega_1)$ or
  $(G,\omega_2)$. 

  Consequently, $P$ has no toric chains except for those of cardinality $0,1,2$,
  that is, the empty set $\varnothing$, the $5$ singletons 
  and the $5$ edge pairs in $G$.  In particular, the maximum size of
  a toric chain is $2$.  From this one can also easily 
  check that the combinatorial toric antichains of $P$ are the 
  empty set $\varnothing$, the $5$ singletons, and the $5$ pairs $\{i,j\}$
  which do not form edges of $G$.  In fact,
  all of these are also geometric toric antichains, so in this
  example the two notions coincide, and for either one the
  toric width is $2$.

  However, as $|V|=5$, there is no partition of $V$ into 
  two toric chains (the analogue of Dilworth's Theorem fails), 
  nor into two toric antichains (the analogue of Mirsky's Theorem 
  fails).
\end{ex}

%%%%%%%%%%%%%%%%%%%%%%%%%%%%%

\bibliographystyle{amsplain}

\providecommand{\bysame}{\leavevmode\hbox to3em{\hrulefill}\thinspace}
\providecommand{\MR}{\relax\ifhmode\unskip\space\fi MR }
\providecommand{\MRhref}[2]{%
  \href{http://www.ams.org/mathscinet-getitem?mr=#1}{#2}
}
\providecommand{\href}[2]{#2}

\end{document}